\documentclass[onefignum,onetabnum]{siamart171218}

\usepackage[margin=1in]{geometry}

\usepackage{amsmath,amsfonts,amssymb}
\usepackage{amsbsy}
\usepackage{bbm}
\usepackage{cases}
\usepackage{centernot}
\usepackage{color}
\usepackage[english]{babel}
\usepackage{epsfig}
\usepackage{graphics}
\usepackage{graphicx}
\usepackage{hyperref}
\usepackage[utf8]{inputenc}        
\usepackage{lipsum}
\usepackage{mathtools}
\usepackage{mdframed}
\usepackage{multicol}
\usepackage{mathrsfs}
\usepackage{url}
\inputencoding{utf8}
\usepackage{verbatim}
\usepackage{stmaryrd}
\usepackage{wasysym}%
\usepackage{scalerel}
\DeclareMathAlphabet{\mathpzc}{OT1}{pzc}{m}{it}

\newcommand{\C}{{\mathbb C}}
\newcommand{\R}{{\mathbb R}}

\newcommand{\norm}[1]{\left\lVert#1\right\rVert}

\newcommand{\sjump}[1]{\left[#1\right]}

\newcommand{\Gop}{{\mathscr G}}
\newcommand{\Sp}{{\mathbb S}}
\newcommand\nxs{\hspace*{-0.2mm}}
\newcommand{\Gr}{{\mathbb G}}
\newcommand{\xhat}{{\hat{\boldsymbol{\xi}}}}
\newcommand{\bd}{\boldsymbol{d}}
\newcommand{\G}{\Gamma}
\newcommand{\Gz}{{\Gamma_0}}
\newcommand{\xbold}{{\boldsymbol{\xi}}}
\newcommand{\ybold}{{\boldsymbol{y}}}

\newcommand{\Fop}{{\mathcal F}}

\newcommand{\bnu}{\boldsymbol{\nu}}

\newcommand{\Hop}{{\mathscr H}}

\newcommand{\ubold}{\bold{u}}
\newcommand{\bn}{\boldsymbol{n}}
\newcommand{\bR}{\bold{R}}
\newcommand{\tbold}{\bold{t}}
\newcommand{\us}{\ubold_{\text{\exs sc}}}
\newcommand{\vbold}{\bold{v}}

\newcommand{\dsy}{\,\,\text{d}S_{\ybold}}
\newcommand{\dy}{\,\,\text{d}V_{\ybold}}

\newcommand{\g}{\bold{g}}
\newcommand{\zbold}{{\boldsymbol{x}}}
\newcommand{\bq}{\boldsymbol{q}}
\newcommand{\bxi}{\boldsymbol{\xi}}
\newcommand{\bx}{\boldsymbol{x}}
\newcommand{\bw}{\boldsymbol{w}}
\newcommand{\bphi}{\boldsymbol{\phi}}
\newcommand{\boldeta}{\boldsymbol{\eta}}

\newcommand{\Top}{{\mathcal T}_R}
\newcommand{\Hh}{H^{1/2}(S_R)^3}
\newcommand{\Hmh}{H^{-1/2}(S_R)^3}

\newcommand{\KK}{{\mathbf K}}
\newcommand{\CC}{{\mathbf C}}
\newcommand{\dnu}{\partial^*_{\bnu}}

\newcommand{\Id}{{\mathbf I}_{3\times 3}}

\newcommand{\bv}{\boldsymbol{v}}
\newcommand{\GG}{\boldsymbol{\G}_{0}}
\newcommand{\Grb}{{\mathbb G}_b}

\newcommand\exs{\hspace*{0.4mm}}
\newcommand{\supp}{\text{supp}\,\,}
\makeatletter
\newsavebox{\@brx}
\newcommand{\llangle}[1][]{\savebox{\@brx}{\(\m@th{#1\langle}\)}%
  \mathopen{\copy\@brx\kern-0.5\wd\@brx\usebox{\@brx}}}
\newcommand{\rrangle}[1][]{\savebox{\@brx}{\(\m@th{#1\rangle}\)}%
  \mathclose{\copy\@brx\kern-0.5\wd\@brx\usebox{\@brx}}}
\makeatother

\newtheorem{assumption}{Assumption}
\numberwithin{assumption}{section}
\usepackage{lipsum}
\usepackage{amsfonts}
\usepackage{graphicx}
\usepackage{epstopdf}
\usepackage{algorithmic}
\ifpdf
  \DeclareGraphicsExtensions{.eps,.pdf,.png,.jpg}
\else
  \DeclareGraphicsExtensions{.eps}
\fi

\newsiamremark{remark}{Remark}
\newsiamremark{hypothesis}{Hypothesis}
\crefname{hypothesis}{Hypothesis}{Hypotheses}
\newsiamthm{claim}{Claim}

\headers{Real-time imaging of interfacial damage in layered composites}{F. Pourahmadian, I. D. Teresa}

\title{Real-time imaging of interfacial damage in layered composites\thanks{Submitted to the editors 02 Jan 2018.
}}

\author{ 
Fatemeh Pourahmadian\thanks{ Department of Civil, Environmental and Architectural Engineering, University of Colorado, Boulder, USA(\email{ fatemeh.pourahmadian@colorado.edu}).}\and 
Irene de Teresa\thanks{Mathematical Sciences, University of Delaware, Newark, Delaware 19716, USA.} }

\usepackage{amsopn}

\ifpdf
\hypersetup{
  pdftitle={Real-time imaging of interfacial damage in layered composites},
  pdfauthor={Irene de Teresa, Fatemeh Pourahmadian}
}
\fi

\begin{document}

\maketitle

\begin{abstract}
A theoretical platform is developed for active elastic-wave sensing of (stationary and advancing) fractures along bi-material interfaces in layered composites. Damaged contact surfaces are characterized by a heterogeneous distribution of (elastic) stiffness $\KK$ which is a-priori unknown. The proposed imaging functional takes advantage of sequential wavefield measurements for non-iterative and concurrent reconstruction of multiple fractures in heterogeneous domains, with an exceptional spatial resolution and with minimal sensitivity to measurement errors and uncertain elasticity of damage zones. This is accomplished through adaptation of the $F_\sharp$-factorization method (FM) applied to elastodynamic wavefields in a sensing sequence i.e.~a measurement campaign performed before and after the formation (or evolution) of interfacial fractures. The direct scattering problem is formulated in the frequency domain where the damage support is illuminated by a set of incident P- and S-waves, and thus-induced scattered waves are monitored at the far field. In this setting, the germane mixed reciprocity principle is introduced, and $F_\sharp$-factorization of the differential scattering operator is rigorously established for composite backgrounds. These results lead to the development of an FM-based fracture indicator whose affiliated cost functional is inherently convex, thus, its minimizer can be computed without iterations. This attribute coupled with the reduced sampling space, proposed in this work, remarkably expedite the data inversion process. The performance of proposed imaging functional is demonstrated by a set of numerical experiments.
\end{abstract}

\pagestyle{myheadings}
\thispagestyle{plain}
%\markboth{I.DE TERESA AND F.POURAHMADIAN}{Active imaging of interfacial damage in layered composites}

%-------------------------------------------------------------------------------------------------------------------------------------------------------------------------------------------------------------------------------------------------------------
\section{Introduction}
%-------------------------------------------------------------------------------------------------------------------------------------------------------------------------------------------------------------------------------------------------------------

\noindent Layered structure is a fundamental feature of a large class of (natural and engineered) materials such as crustal rocks~\cite{Narr2011}, composites (e.g.~concretes and geopolymers)~\cite{zaki2015,nour2016}, biological and biomimetic tissues~\cite{shim2002}, metamaterial transducers, lenses, and cloaks~\cite{Shi2015,nour2016,stenger2012}. In geomaterials, active interfacial mechanisms play a key role in~(a)~formation and evolution of aquifers and hydrocarbon reservoirs~\cite{Narr2011},~(b)~propagation of faults and generation of earthquakes~\cite{Hedayat2014,fatemeh-thesis},~(c)~progressive failure along discontinuities resulting in e.g.,~catastrophic failure of civil infrastructure~\cite{wisnom2010}. In biomaterials, the \emph{continuity and strength of interfaces} is important~(a)~in multi-compartment scaffolds for efficient cellular transport between scaffold phases~\cite{thom2012,shim2002},~(b)~in tumor-tissue contacts for timely detection of mutation and tumor branching~\cite{yin2015,gaed2008}, and~(c)~in metal-to-bone junctions in orthopedic total joint implants for effective load transfer and for prevention of pain and unprecedented failure due to stress concentration~\cite{thom2012}. In metamaterials, their periodic (or layered) microstructure allows for enhanced heat/electric conductivity as well as acoustic/elastic waveguiding effects~\cite{stenger2012,mash2013}. However, delamination and other interfacial variations are major impediments to effective performance of such systems. Therefore, real-time monitoring and characterization of interfaces in layered composites has a major impact in a wide range of engineering applications.           

\noindent Traditionally, elastic waves are deployed for internal characterization of solids thanks to their interaction with hidden anomalies~\cite{Song2007,Li2014,solo2016,pour2015}. Common \emph{passive} approaches to remote sensing of damage in material interfaces e.g.~acoustic emission~\cite{zaki2015} are reliant upon significant assumptions on the nature of wave motion in the background domain~\cite{shap2015}, and their usage is temporary i.e.~fairly restricted to the time-span of interfacial instability and evolution. Additionally, these methods are agnostic to~{{(a)}}~normal mode of fracture propagation~\cite{Maje2007}, and~{{(b)}}~creep-like (aseismic) advancement of discontinuities~\cite{scot1994,scot1994(2),calo2011}. The more general waveform tomography methods~\cite{mass2017,calo2011,feigl2017} are, by and large, iterative (cost-inefficient) and thus inapplicable in a real-time sensing framework. Also, the spatial resolution of the final image in these schemes are mostly dependent upon the frequency regime of illumination. Other diagnostic technologies that make use of electromagnetic (or other non-mechanical) waves e.g.~CT and MRI~\cite{gaed2008,Will2016} are mostly focused on laboratory applications, and their implementation for large systems is quite costly and challenging.

\noindent Recently, a suit of imaging tools rooted in Kirsch's Factorization Method (FM)~\cite{kirsch, kirschFM} are developed, mostly in the context of electromagnetic inverse scattering, which are non-iterative and flexible in terms of sensing configuration~\cite{pour2017,fatemeh,delam,max-delam}. The key idea is to construct an indicator functional whose support is precisely the loci of hidden anomalies in the domain of interest. For damage reconstruction, the imaging indicator is computed over a designated sampling grid, where the support of its highest values identifies the location and geometry of the sought features. On repeated evaluation of the indicator field over a span of time -- that is fast thanks to the non-iterative nature of calculations, one may retrieve the spatiotemporal evolution of damage almost in real time.    

\noindent In this study, FM is generalized for active reconstruction of interfacial damage (e.g.~corrosion, fatigue fractures, delamination) in layered composite materials. Here, each layer (i.e.~component) is considered as a linear, isotropic, elastic solid with distinct material properties. In the reference (undamaged) state, the layers are connected through continuous i.e.~welded interfaces, a subset of which will undergo damage (or evolution) over time. As a result, in the secondary configuration, there exist some discontinuous bi-material interfaces formulated via the well-known linear slip model~\cite{Pyrak1992,pour2017(3)}, where the traction applied to an interface is linearly related to the affiliated jump in the displacement wavefield via the so-called stiffness matrix. It is worth mentioning that more general  {\em approximate transmission conditions} have been derived (in the context of acoustics and electromagnetism) by using perturbation techniques under the premise of vanishing damage thickness at the contact interface (e.g. ~\cite{MaxHaddar, haddar08, chun, ATCHaddar, Delourme2, delam}). Thus-obtained boundary conditions typically involve surface differential operators and achieve higher order accuracy in modeling heterogeneous contacts, which can be naturally extended to elastic interfaces.

\noindent  The proposed imaging indicator will be developed by analyzing the range of the so-called differential scattering operator, constructed on the basis of~(a)~sequentially measured remote data, and~(b)~synthetic waveforms in the background (undamaged) domain. The resulting imaging functional will be evaluated over the intrinsic bi-material interfaces of the background domain, minimizing the number of sampling points necessary to detect advancing anomalies along such interfaces. This 3D sensing tool is fast (non-iterative) and capable of simultaneously reconstructing multiple fractures. In addition, this indicator can provide {\emph{high-resolution images}} of damage irrespective of the illumination frequency~\cite{kirsch}, while offering a significant flexibility in terms of the location and number of sensors i.e.~sources and receivers~\cite{fatemeh-thesis}.

\noindent This paper is organized as the following:~Section~\ref{sec1} describes the problem under consideration and introduces the preliminary concepts associated to the forward scattering problem; Section~\ref{sec2} establishes the well-posedness of the direct problem required for rigorous justification of the imaging indicator (in layered composites) developed in Section~\ref{sec3}. The performance of thus-obtained reconstruction tool is then demonstrated through a set of numerical experiments in Sections~\ref{sec4}.

\subsection{Problem Statement} \label{sec1}
%-------------------------------------------------------------------------------------------------------------------------------------------------------------------------------------------------------------------------------------------------------------
This work aims to develop an inverse solution for active geometric (shape and size) reconstruction of fractures propagating along bi-material interfaces. In this vein, scattering of elastic (P- and S-) waves by a bounded penetrable obstacle $\Omega\subset\R^3$, featuring a layered internal structure and Lipschitz external boundary $\G_1$, is considered. Here, $\Omega$ is composed of two distinct layers, $\Omega_-$ and $\Omega_+$, that are closed and Lipschitz-continuous with common interface $\G$, housed in a linear, homogeneous, and isotropic elastic solid $\Omega_{ext}:=\R^3\setminus\overline{\Omega}$.
\begin{figure}[tp]
\center\includegraphics[width=0.7\linewidth]{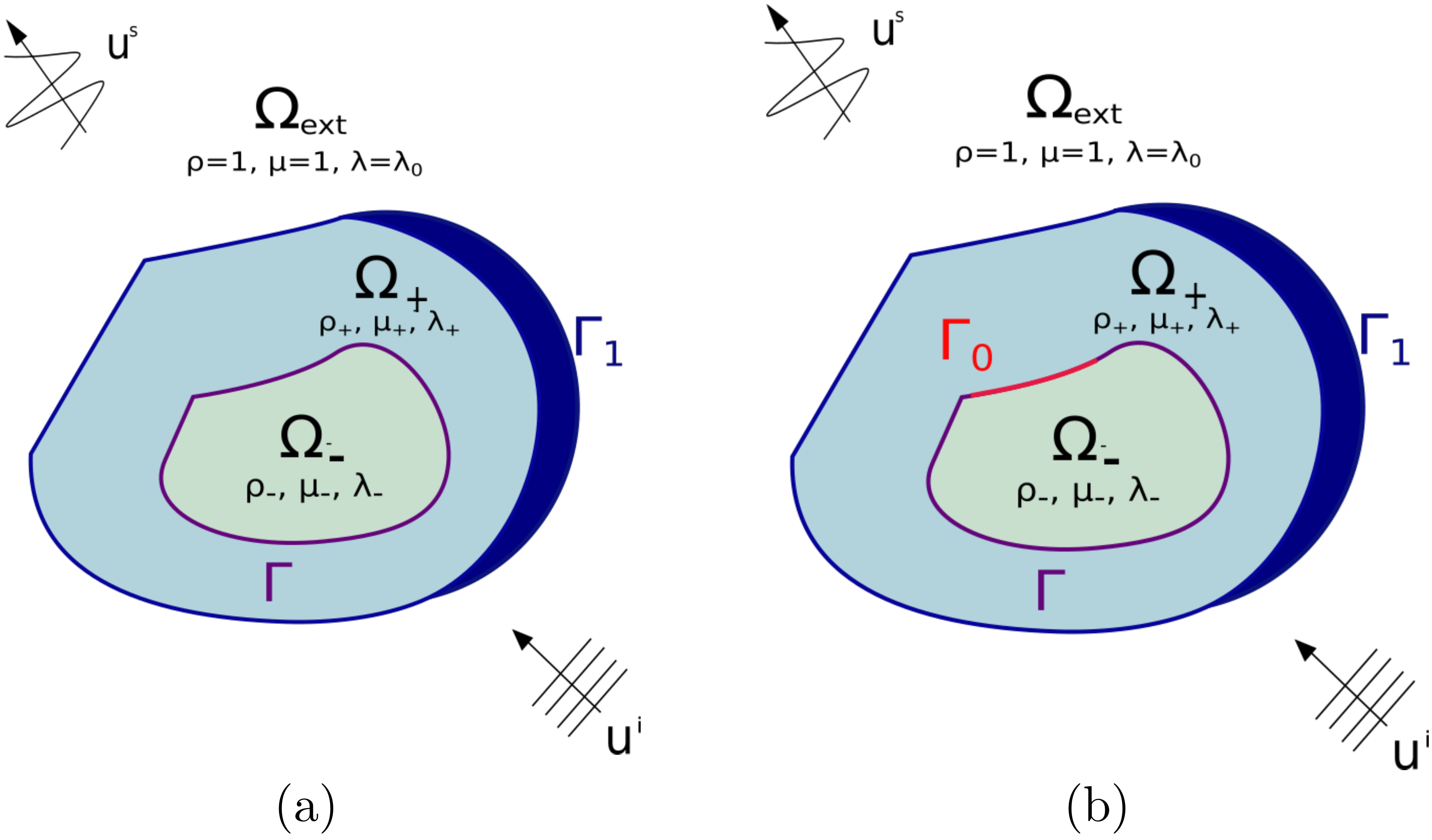} 
\caption{\small Scattering of plane (P- and S-) waves in a layered domain:~(a)~transversal cut of the background (undamaged) domain, consisting of three layers $\Omega_-$, $\Omega_+$, and $\Omega_{ext}$, and (b) its affiliated damaged configuration, where the elastic interface $\Gz$ appears at the bi-material interface $\G$.} \label{fig_elastic1}
\end{figure} 
In the \emph{background} domain shown in Fig.~\ref{fig_elastic1}~(a), bimaterial interfaces $\G_1$ and $\G$ are intact and the continuity condition applies for both the displacement and the traction. At a later stage, however, a subset of $\G \cup \G_1$ is evolved into an interfacial damage zone i.e., elastic interface $\Gz\subset\G \cup \G_1$, in the so-called damaged configuration. According to Fig.~\ref{fig_elastic1}~(b), $\Gz$ is an open surface with relative boundary $\partial \Gz$ on $\G$. The three layers $\Omega_+$, $\Omega_-$, and $\Omega_{ext}$ are distinct in terms of their material properties, and characterized by their density $\rho_\kappa$ and elasticity tensor $\CC^\kappa$, normalized via the respective properties of $\Omega_{ext}$, with $\kappa \in \lbrace +, -, 0 \rbrace$ respectively indicating the domains $\Omega_+$, $\Omega_-$, and $\Omega_{ext}$ . The fourth-order elasticity tensor $\CC^\kappa$, specifying the linear constitutive relation between elastic strain and stress tensors, is given as the following~\cite{segel},
\begin{equation}\label{CC}
\CC^\kappa_{ijk\ell}= \lambda_\kappa \delta_{ij}\delta_{\ell k} + \mu_\kappa (\delta_{i\ell}\delta_{jk}+\delta_{ik}\delta_{j\ell}), \quad {i,j,k,\ell} \in \lbrace 1,2,3 \rbrace, \quad \kappa \in \lbrace +, -, 0 \rbrace,
\end{equation}
where $\mu_\kappa, \lambda_\kappa\!>\!0$ represent the Lam\'{e}'s parameters. In light of the earlier remark on scaling, $\rho_0,\,\lambda_0$ and $\mu_0$ of the exterior domain $\Omega_{ext}$ are considered constant with values $\rho_0 = 1$, $\mu_0 = 1$, and $\lambda_0 = 2 \nu_0/(1-2 \nu_0)$ where $\nu_0$ is the Poisson's ratio. The respective material constants in $\Omega_-$ and $\Omega_+$ are known \emph{piece-wise smooth} functions, with possible jumps at the interfaces $\G,\,\G_1$ satisfying:
\begin{equation}\label{mon-cond}
\begin{aligned}
&(\lambda_+ -\, \lambda_-)(\mu_+ -\, \mu_-)\ge 0, \qquad &\text{on} \, \G, \\*[0.5 mm]
&(\lambda_+ -\, \lambda_0)(\mu_+ -\, \mu_0)\ge 0, \qquad &\text{on} \, \G_1,
\end{aligned}
\end{equation} 
where $\lambda_\pm,\mu_\pm$ are the Lam\'{e} coefficients $\lambda,\mu$ in $\Omega_\pm$, respectively.\\
\begin{remark}
The monotonicity condition (\ref{mon-cond}) for the Lam\'{e} parameters in transmission problems is necessary in order to have a unique representation of the fields in terms of single- and double-layer potentials in the context of Lipschitz domains (see \cite{escauriaza}). 
\end{remark}

Next, the forward elastic-wave scattering problem is formulated in the frequency domain. To this end, let $\Sp^2$ denote the unit sphere centered at the origin. For a given triplet of vectors $\bd\in\Sp^2$ and~$\bq_p,\bq_s\!\in\mathbb{R}^3$ such that $\bq_p\!\parallel\bd$ and~$\bq_s\!\perp\!\bd$, the domain is illuminated by a combination of compressional and shear plane waves 
\begin{equation}\label{plwa}
\ubold^i(\bxi) ~=~ \bq_p \exs e^{i k_p \bxi \cdot \bd} \:\oplus\: \bq_s \exs e^{i k_s \bxi \cdot \bd}, \qquad \bxi \in \mathbb{R}^3
\end{equation}
propagating in direction~$\bd$, where
\begin{equation}\label{ks_kp}
k_s^2 = \frac{\omega^2}{\mu_0}\quad\text{and}\quad k_p^2=\frac{\omega^2}{\lambda_0 + 2\mu_0},
\end{equation}
denote the respective wave numbers, with $\omega$ being the excitation frequency. The interaction of $\ubold^i$ with $\Omega = \Omega_+ \! \cup \Omega_-$ in the damaged configuration gives rise to the total field $\ubold\in H^1_{loc}(\R^3\setminus\overline{\Gz})^3$  is governed by
\begin{equation} \label{pr1}
\begin{aligned}
&\Delta^* \ubold(\bxi) \,+\, \rho \omega^2 \ubold(\bxi) \,=\, \bold{0}, \quad &  \bxi \in \R^3\!\setminus\! \{\overline{\G \cup \G_1}\}, \\*[0.5 mm]
&\dnu \ubold^-(\bxi) \,+\, \dnu \ubold^ +(\bxi) \,=\, \bold{0}, \, & \bxi \in \G, \\*[0.25 mm]
&\dnu \ubold^-(\bxi) \,=\, \KK(\bxi)[\ubold](\bxi), \,& \bxi \in \Gz, \\*[0.25 mm]
&\ubold^-(\bxi) \,=\, \ubold^+(\bxi),  \,& \bxi \in \G \backslash \overline{\Gz}, \\*[0.5 mm]
&\dnu \ubold^+(\bxi) \,+\, \dnu \ubold^\circ(\bxi) \,=\, \bold{0}, \quad \ubold^+(\bxi) \,=\, \ubold^\circ(\bxi),  \,& \bxi \in \G_1, 
\end{aligned}
\end{equation}
where $\Delta^*$ stands for the Lam\'{e} operator defined by
\begin{equation}\label{D*}
\Delta^*\ubold\nxs:=\nxs \nabla \cdot (\CC(\bxi)\!:\! \nabla \ubold)\nxs=\nabla\cdot \big[ 2\mu(\bxi)\, \boldsymbol{\varepsilon}(\ubold) + \lambda(\bxi) \nabla\!\cdot\!\ubold \,\exs \Id \big], \,\, \boldsymbol{\varepsilon}(\ubold)\nxs:=\nxs \frac{1}{2} \big(\nabla\ubold + \hspace{-0.75mm} \nabla\ubold^{T} \big),
\end{equation}
with $\boldsymbol{\varepsilon}(\ubold)$ introducing the strain tensor for future reference; $\ubold^+$, $\ubold^-$, and $\ubold^\circ$ denote the total field $\ubold$ respectively in $\Omega_+$, $\Omega_-$, and $\Omega_{ext}$; also,
\begin{equation} \label{nD}
\dnu\ubold^\kappa(\bxi):= \bnu(\bxi) \cdot \CC^\kappa(\bxi) : \nabla \ubold^\kappa(\bxi), \qquad  \kappa \in \lbrace  \circ, +, - \rbrace,    \qquad \bxi \rightarrow \G \cup \G_1,
\end{equation}
is the co-normal derivative (or ``traction") along the bi-material interfaces, wherein $\bnu$ defines the outward unit normal vector; $[\ubold](\bxi):=\ubold^+(\bxi)-\ubold^-(\bxi)$ is the jump in displacement field across $\bxi \in \Gz$. In (\ref{pr1}), the interactions between the two faces of damaged interface $\Gz$, due to e.g.~corrosion and surface roughness as a by-product of fracturing, is modeled by the common linear slip interfacial condition~\cite{fatemeh},      
\begin{equation}
\dnu \ubold(\bxi) \,=\, \KK(\bxi) [\ubold](\bxi), \, \qquad  \bxi \in \Gz,
\end{equation}
where $\KK\in L^\infty(\Gz)^{3\times3}$ is the heterogeneous interfacial stiffness matrix, assumed to be symmetric light of the reciprocity principle. To complete the problem statement (\ref{pr1}), consider the decomposition of the total field $\ubold$ in terms of the known incident wavefield $\ubold^i\in H^1_{loc}(\R^3)^3$, according to (\ref{plwa}), and the scattered waveform $\ubold^{sc}\in H^1(\R^3\setminus\overline{\Gz})^3$ i.e.~$\ubold=\ubold^{sc} + \ubold^i$. Thus-obtained scattered field $\ubold^{sc}$ is further decomposed in $\Omega_{ext}$, into P-wave $\ubold^p$ and S-wave $\ubold^s$ components as the following:
\begin{equation}\label{ups}
\begin{aligned}
&\ubold^p(\bxi)\!:=\!\frac{1}{k^2_s - k^2_p} (\Delta + k^2_s)\ubold^{sc}(\bxi), & \qquad  \bxi \in \Omega_{ext},\\
&\ubold^s(\bxi)\!:=\!\frac{1}{k^2_p - k^2_s} (\Delta + k^2_p)\ubold^{sc}(\bxi), & \qquad  \bxi \in \Omega_{ext}. 
\end{aligned}
\end{equation}
In this setting, the following Kupradze radiation conditions~\cite{kupradze}, imposed on the scattered field $\ubold^{sc}$, complements (\ref{pr1}) and completes the problem statement,
\begin{equation}\label{kup-cond}
\begin{aligned}
&\frac{\partial \ubold^p}{\partial r}-ik_p\ubold^p \,=\, o(r^{-1}), \\*[0.5 mm]
&\frac{\partial \ubold^s}{\partial r}-ik_p\ubold^s \,=\, o(r^{-1}), \quad \text{as} \,\, r \rightarrow \infty,
\end{aligned}
\end{equation}
where $r=|\bxi|$ and the limits are uniform with respect to $\hat\bxi=\bxi/|\bxi|$.

%------------------------------------------------------------------------------------------------------------------------------------------------------------------
\section{Well-posedness of the forward problem}\label{sec2}
%------------------------------------------------------------------------------------------------------------------------------------------------------------------

This section investigates the well-posedness of direct scattering problem (\ref{pr1})-(\ref{kup-cond}) in a layered medium, by rigorously extending the related theorems in homogeneous domains~\cite{fatemeh}. This is accomplished within the framework of Fredholm theory where the field equations (\ref{pr1})-(\ref{kup-cond}) are recast in the variational form in terms of $\ubold\in H^1(B_R\setminus\overline{\G}_0)^3$ within an arbitrary ball $B_R\subset \R^3$ of radius $R>0$ such that $\overline{\Omega}\subset B_R$. Thus, one obtaines
\begin{eqnarray}
A(\ubold,\vbold) ~=~ {\mathcal L}(\vbold), \,\qquad \forall \vbold , \ubold\in H^1(B_R\setminus\overline{\G}_0)^3,\label{varform}
\end{eqnarray}
where
\begin{equation}\label{def-A}
\begin{aligned}
A(\ubold,\vbold)~=~&\int_{B_R \backslash \Gz} \!\!\! \overline{\nabla\vbold}:\CC:\nabla\ubold\dy \,-\,  \omega^2 \int_{B_R \backslash \Gz} \rho\, \overline{\vbold}\cdot\ubold\,\dy \,\,+ \\*[1 mm]
 +\, & \int_{\Gz} [\overline{\vbold}] \cdot \KK \left[ \ubold\right] \dsy \,-\, \int_{ S_R}  \overline{\vbold} \cdot \Top \ubold \dsy,\\*[2 mm]
 {\mathcal L}(\vbold)~=~&\int_{ S_R}\Big\{\overline{\vbold} \cdot \dnu \ubold^i - \overline{\vbold} \cdot \Top \ubold^i \Big\} \dsy.
 \end{aligned}
\end{equation}
Here, $S_R$ represents the boundary of $B_R$, and $\Top:\Hh\rightarrow\Hmh$ is a Dirichlet-to-Neumann (DtN) operator
\begin{equation}\label{DtN}
\Top (\bphi)(\bxi) ~\colon \!\!\! =~ \dnu \ubold_\phi (\bxi), \qquad  \bxi \in { S_R},
\end{equation}
where $\ubold_\phi$ satisfies
\begin{equation}
\begin{aligned}
&\Delta^*\ubold_\phi(\bxi) \,+\, \omega^2\ubold_\phi(\bxi) ~=~ \bold{0}, \qquad   &\bxi \in  \R^3\!\setminus\!\overline{B_R},\\*[1 mm]
&\ubold_\phi(\bxi)~=~\bphi(\bxi), \qquad & \bxi \in  S_R,
\end{aligned}
\end{equation}
complemented with the Kupradze radiation conditions at infinity~\cite{kupradze}, similar to (\ref{kup-cond}). It should be noted that the DtN operator $\Top$ is well-posed and bounded~\cite{CC06}. For clarity of discussion, let us define the following function spaces.
\begin{equation}
\begin{aligned}
& H^{\pm \frac{1}{2}}(\Gz) := \Big\{u|_{\Gz}\colon~u\in H^{\pm \frac{1}{2}}(\G)\Big\},\\*[1mm]
& \widetilde{H}^{\pm \frac{1}{2}}(\Gz) := \Big\{u\in H^{\pm \frac{1}{2}}(\G)\colon~\supp\!\!(u) \subset \overline\Gz\Big\},
\end{aligned}
\end{equation}
wherein $\supp\!\!(u)$ is the \emph{essential} support of $u$ defined as the largest relatively closed subset of $\G$ such that $u=0$ almost everywhere in $\G \!\setminus\! \supp\!\!(u)$. One may note that $H^{1/2}(\Gz)$ and $\widetilde{H}^{1/2}(\Gz)$ can be considered Hilbert spaces, providing that they are endowed with a restricted $H^{1/2}(\G)-$inner product. Additionally, one may recall that $H^{-1/2}(\Gz)$ and $\widetilde{H}^{-1/2}(\Gz)$ are respectively the dual spaces of $\widetilde{H}^{1/2}(\Gz)$ and $H^{1/2}(\Gz)$. Accordingly, the following embeddings hold
\begin{eqnarray}\label{embed}
&\widetilde{H}^{1/2}(\Gz)\subset H^{1/2}(\Gz)\subset L^2(\Gz)  \subset \widetilde{H}^{-1/2}(\Gz) \subset H^{-1/2}(\Gz).&
\end{eqnarray}

\begin{remark}{\em 
If the total field $\ubold \in H^1(B_R\setminus\overline{\G}_0)^3$, then the displacement jump $\sjump{\ubold}\in \widetilde{H}^{1/2}(\Gz)^3$. Also, if the heterogeneous stiffness matrix $\KK\in L^\infty(\Gz)^{3\times3}$, then $\KK\sjump{\ubold}\in H^{-1/2}(\Gz)^3$ (see Corollary 8.8 in \cite{kress}). In this setting, the term 
$$\int_{\Gz} [\overline{\vbold}](\text{\bf y}) \! \cdot  \KK(\text{\bf y}) \! \left[ \ubold\right]\!(\text{\bf y}) \dsy$$
in \eqref{def-A} can be interpretted as a duality pairing $\langle\cdot,\cdot\rangle_{H^{-1/2}(\Gz)^3,\widetilde{H}^{1/2}(\Gz)^3}$, pivoting with respect to the $L^2(\Gz)^3$ inner product. In what follows, $\Re(\cdot)$ (resp.~$\Im(\cdot)$) stands for the real (resp.~imaginary) part of its argument.
}
\end{remark}
\begin{lemma}\label{lemmaDtN}
The DtN operator $\Top$ can be decomposed as $\Top=\Top^c+\Top^0$, where $\Top^c:\Hh\rightarrow\Hmh$ is compact, and  $-\Top^0:\Hh\rightarrow\Hmh$ is non-negative and self-adjoint. Moreover,
\begin{eqnarray} \label{lmma2.1}
&\Im \langle \Top (\boldsymbol{\phi}),\,\boldsymbol{\phi}\rangle_{S_R} >0, \qquad  \forall  \boldsymbol{\phi}\in \Hh\!\setminus\!\{\boldsymbol{0}\}.&
\end{eqnarray}
\end{lemma}
\begin{proof}
The argument is similar to the proof of Lemma~1~in~\cite{fatemeh}, where $\Top^0$ is defined according to the Riesz representation theorem by
$$ \langle \Top^0 (\boldsymbol{\phi}),\,\boldsymbol{\psi}\rangle_{S_R}\,:=\, - \int_{B_{R_{0}} \!\setminus B_R } \!\!\! \overline{\nabla\ubold_{\boldsymbol{\psi}}}:\CC:\nabla\ubold_{\boldsymbol{\phi}}\dy,$$
with $B_{R_{0}}$ being a ball of radius $R_0 > R$.
\end{proof}

\begin{theorem}(well-posedness of the forward scattering problem)
Providing that $\KK\in L^\infty(\Gz)^{3\times 3}$ and $\Im([\overline{\vbold}]\!\cdot\!\KK[\vbold])\leq 0$,  $\forall [\vbold]\in \tilde{H}^{1/2}(\Gz)^3$, the variational form (\ref{varform}) is well-posed. In other words,~$\forall\ubold^i \in H^1_{loc}(\R^3)^3$, there exists a unique solution $\ubold\in {H}^1(B_R\setminus\overline{\G}_0)^3$ that \emph{continuously} depends on $\ubold^i$.
\end{theorem}
\begin{proof}
Let us decompose $A$, in the variational form (\ref{varform}), as follows: 
$$A(\ubold,\vbold)~=~A_0(\ubold,\vbold)\,+\,B(\ubold,\vbold), \qquad \forall\ubold,{\vbold}\in H^1(B_R\setminus\overline{\G}_0)^3,$$
\begin{equation}\nonumber
\begin{aligned}
&A_0(\ubold,\vbold)~=\,\int_{B_R\setminus\overline{\G}_0} \overline{\nabla\vbold}:\CC:\nabla\ubold\dy \,+\, \omega^2 \int_{B_R\setminus\overline{\G}_0} \overline{\vbold}\cdot\ubold\dy 
\, - \int_{ S_R} \Top^0(\ubold)\cdot\overline{\vbold} \dsy, \\*[1 mm]
& B(\ubold,\vbold)~=\, -\,\omega^2\! \int_{B_R\setminus\overline{\G}_0}\!(1+\rho) \overline{\vbold}\cdot\ubold\dy
\,+\int_{\Gz} [\overline{\vbold}] \cdot \KK \left[ \ubold\right]\dsy \,- \int_{ S_R} \overline{\vbold} \cdot \Top^c (\ubold) \dsy.
\end{aligned}
\end{equation} 
By invoking the Korn's inequality~\cite{mclean} and in view of (\ref{D*}), one obtains
\begin{equation} \nonumber
\begin{aligned}
|A_0(\ubold,\ubold)| &~=\,\int_{B_R\setminus\overline{\G}_0} \Big\lbrace 2 \mu  |\boldsymbol{\varepsilon(\ubold)}|^2 \,+\, \lambda  |\nabla\!\cdot\! \ubold|^2 +  \omega^2  |\ubold|^2 \Big\rbrace \dy 
\,-\, \int_{ S_R} \Top^0 (\ubold)\cdot\overline{\ubold} \dsy \geq \\*[1 mm]
&\,\,\,\geq \int_{B_R\setminus\overline{\G}_0} \Big\lbrace 2 \mu  |\boldsymbol{\varepsilon(\ubold)}|^2 \,+  \omega^2  |\ubold|^2 \Big\rbrace \dy 
\,\geq\, C \norm{\ubold}^2_{H^1(B_R\setminus\overline{\G}_0)^3},  
\end{aligned}
\end{equation}
where $C>0$ is a constant independent of $\ubold$. This implies that $A_0:H^1(B_R\setminus\overline{\G}_0)^3\times H^1(B_R\setminus\overline{\G}_0)^3\rightarrow\C$ is coercive. Moreover, for all $\ubold,\vbold\in H^1(B_R\setminus\overline{\G}_0)^3$,
\begin{equation} \label{Buv}
\begin{aligned}
 |B(\ubold,\vbold)|\,\leq\,&\,\omega^2\,(1+\norm{\rho}_\infty) \norm{\ubold}_{L^2(B_R\setminus\overline{\G}_0)^3} \norm{\vbold}_{L^2(B_R\setminus\overline{\G}_0)^3} \\*[1 mm]
 &+\, \norm{\KK}_\infty \norm{[\ubold]}_{L^2(\Gz)^3}\norm{[\vbold]}_{L^2(\Gz)^3} \,+\,  \norm{\Top^c (\ubold)}_{H^{-\frac{1}{2}}( S_R)^3}\norm{\vbold}_{H^{\frac{1}{2}}( S_R)^3}.
 \end{aligned}
\end{equation}
Thus, for all test functions ${\vbold}$ where $\norm{\vbold}_{H^1(B_R\setminus\overline{\G}_0)^3}=1$, (\ref{Buv}) can be recast as
\begin{equation}
 |B(\ubold,\vbold)|\leq C_0 \norm{\ubold}_{L^2(B_R\setminus\overline{\G}_0)^3} + C_1 \norm{\KK}_\infty \norm{[\ubold]}_{L^2(\Gz)^3} + C_2 \norm{\Top^c (\ubold)}_{\Hmh},
\end{equation}
where $C_0=\omega^2(1+\norm{\rho}_\infty)$; $C_1$ is twice the norm of the trace operator from $H^1(B_R\setminus\overline{\G}_0)^3 \rightarrow L^2(\G_0)^3$; and $C_2>0$ stands for the norm of the compact  trace operator from $H^1(B_R\setminus\overline{\G}_0)^3 \rightarrow H^{1/2}( S_R)^3$. Now, let $\{\ubold_n\}$ be a sequence that converges weakly to $\bold{0}$ in $H^1(B_R\setminus\overline{\G}_0)^3$, then, from the compactness of $H^1(B_R\setminus\overline{\G}_0)^3\subset L^2(B_R\setminus\overline{\G}_0)^3$, the boundedness of the trace operator $H^1(B_R\setminus\overline{\G}_0)^3 \rightarrow H^{1/2}(S_R \cup {\G}_0)^3$, together with the compactness of the embedding $\widetilde{H}^{1/2}(\G_0)^3\subset L^2(\G_0)^3$, and the compactness of $\Top^c:\Hh\rightarrow\Hmh$, we conclude that $B(\ubold_n,\vbold) \rightarrow 0$ for all $\vbold$ in $H^1(B_R\setminus\overline{\G}_0)^3$. Hence, $B(\cdot,\cdot)$ is a compact sesquilinear form.
Therefore, $A$ is the sum of a coercive and a compact sequilinear form, and thus, the forward scattering problem is of Fredholm type, i.e., (\ref{pr1})-(\ref{kup-cond}) is well-posed providing that its solution is unique. The latter is established in the sequel. \\
Let us suppose that ${\mathcal L}(\ubold)=\boldsymbol{0}$ in (\ref{varform}), then the imaginary part of $A(\ubold,\ubold)$, according to (\ref{def-A}), takes the following form:
\begin{equation} \label{Imu}
 \int_{\Gz} \Im( \KK \left[ \ubold\right]\cdot[\overline{\ubold}])\dsy - \Im\Big\{\int_{ S_R} \Top (\ubold)\cdot\overline{\ubold}\dsy\Big\} ~=~0.
\end{equation}
Invoking (\ref{lmma2.1}) and recalling the negative definiteness of stiffness matrix $\KK(\bxi)$, from the theorem statement, one concludes from (\ref{Imu}) that $\ubold=\bold{0}$ on $S_R$. Therefore, $\ubold(\bxi)=\bold{0}$ in $\bxi \in B_R$, via the unique continuation principle, which completes the proof.  
\end{proof}
%-----------------------------------------------------------------------------------------------------------------------------------------------------------------------------------------------------------
\section{Inverse solution}\label{sec3}
%-----------------------------------------------------------------------------------------------------------------------------------------------------------------------------------------------------------
This section aims to develop a robust and active imaging algorithm to reconstruct the damaged portion $\Gz$ of bimaterial interfaces $\G \exs \cup \exs \G_1$ using measured scattered data $\ubold^{sc}$ corresponding to every incident P- and S-wave propagating in direction $\bd\in\Sp^2$ and polarized along $\bq\exs\in\exs\R^3$. 
The full-waveform inversion will be performed non-iteratively and in a particularly expeditious manner. This is accomplished by rigorously adapting the Kirsch's Factorization Method (FM)~\cite{kirsch}, originally developed in the context of electromagnetism, for imaging fractures in layered composites by way of elastic waves. The ensuing developments are performed in the frequency domain under the premise that the background domain, i.e.~the configuration in absence of the sought-for damage $\Gz$, is given in terms of its structure and material properties. In other words, it is assumed that $\Omega_{ext}(\mu_0, \rho_0, \lambda_0)$, $\Omega_{+}(\mu_+, \rho_+, \lambda_+)$, and $\Omega_{-}(\mu_-, \rho_-, \lambda_-)$ are known a-priori. This knowledge, particularly the intact topology of bimaterial interfaces, will be used within FM framework to minimize the number of sampling points required to locate and characterize advancing interfacial damages.
\vspace{-2 mm}
%-----------------------------------------------------------------------------------------------------------------------------------------------------------------------------------------------------------
\subsection{Preliminary concepts}
%-----------------------------------------------------------------------------------------------------------------------------------------------------------------------------------------------------------

%~~~~~~~~~~~~~~~~~~~~~~~~~~~~~~~~~~~~~~~~~~~~~~~~~~~~~~~~~~~~~~~~~~~~~~~~~~~~~~~~~~~~~~~~~~~~~~~~~~~~~~~
{\it Elastic Herglotz wave function.}~For a given density $\g \in L^2(\Sp^2)^3$, we consider the unique decomposition 
 \vspace{-2 mm}
\begin{equation}\label{herden}
\g \;:=\; \g_p \oplus\, \g_s, 
\end{equation}
such that $\g_p(\bd)\!\parallel\!\bd$ and $\,\g_s(\bd)\!\perp\!\bd$, $\,\bd\in\Sp^2$. In dyadic notation, one has 
\begin{equation}\label{freef}
\g_p(\bd) := (\bd \nxs \otimes \nxs \bd \exs) \cdot \g(\bd) \quad ~~\text{and}~~ \quad   \g_s(\bd) := (\Id - \bd \nxs \otimes \nxs \bd \exs) \cdot \g(\bd).
\end{equation}
In this setting, the elastic Herglotz wave function~\cite{Dassios1995} is defined as 
\begin{equation}\label{HW}
\ubold_\g(\bxi) ~: =~  \int_{\Sp^2} \g_p(\bd) \exs  e^{i k_p \bd \cdot \bxi} \,\, \text{d}S_{\bd} ~\; \textcolor{black}{\oplus} \, \int_{\Sp^2} \g_s(\bd) \exs  e^{i k_s \bd \cdot \bxi} \,\, \text{d}S_{\bd},  \qquad \bxi \in \R^3, 
\end{equation}
in terms of {the compressional and shear wave densities~$\g_p$ and~$\g_s$, respectively.

\vspace{1 mm}
\noindent {\it Incident plane-wave tensor.}~In the case of single-incident illumination where the domain is excited by a plane wave propagating in direction $\bd \in \Sp^2$ and polarized along $\bq := \bq_p \oplus \bq_s$, the incident plane-wave tensor $\bold{W}^i(\cdot,\bd)\in {\mathcal C}^\infty(\R^3)^{ 3\times3}$ is defined by
\begin{equation}\label{Wi}
\bold{W}^i (\bxi,\bd) \,:=\, e^{ik_s\bxi\cdot\bd}(\Id-\bd\otimes\bd) \,+\, e^{ik_p\bxi\cdot\bd}\, \bd\otimes\bd, \qquad \bxi \in \R^3,
\end{equation}
where $k_s$ and $k_p$ are given in (\ref{ks_kp}). In view of (\ref{plwa}) and (\ref{freef}), one may note that 
\begin{equation}
 \ubold_\g(\bxi) ~=~  \int_{\Sp^2} \bold{W}^i (\bxi,\bd)  \g (\bd) \,\, \text{d}S_{\bd}, \qquad  \ubold^i(\bxi) ~=~ \bold{W}^i (\bxi,\bd)\,\bq ,   \,\, \forall \bxi \in \R^3, \, \bd \in \Sp^2.
\end{equation}

\noindent {\it Background wavefield.}~The interaction of incident wavefield $\ubold^i$ with the layered background domain, Fig.~\ref{fig_elastic1}~(a),~gives rise to the total field $\ubold_b \in H^1_{loc}(\R^3)^3$ governed by
\begin{equation} \label{pr1bb}
\begin{aligned}
&\Delta^* \ubold_b(\bxi) \,+\, \rho(\bxi) \exs \omega^2 \ubold_b(\bxi) \,=\, \bold{0}, \quad &  \bxi \in \R^3\!\setminus\! \{\overline{\G \cup \G_1}\}, \\*[0.6 mm]
&\dnu \ubold_b^-(\bxi) \,+\, \dnu \ubold_b^ +(\bxi) \,=\, \bold{0}, \quad \ubold_b^-(\bxi) \,=\, \ubold_b^+(\bxi), \, & \bxi \in \G, \\*[0.7 mm]
&\dnu \ubold_b^+(\bxi) \,+\, \dnu \ubold_b^\circ(\bxi) \,=\, \bold{0}, \quad\, \ubold_b^+(\bxi) \,=\, \ubold_b^\circ(\bxi),  \,& \bxi \in \G_1,
\end{aligned}
\vspace{-1 mm}
\end{equation} 
where the applied traction $\dnu\ubold_b^\kappa(\bxi)$ over $\G \cup \G_1$ for $\kappa \in \lbrace  \circ, +, - \rbrace$ is understood in the sense of \eqref{nD}.  
The total field $\ubold_b=\ubold^i + \ubold_b^{sc}$ involves two contributions, namely:~I)~incident field $\ubold^i=\bold{W}^i\,\bq$, and~II)~scattered field $\ubold_b^{sc}$ satisfying the pertinent Kupradze radiation condition, similar to (\ref{kup-cond}).

\vspace{1 mm}
\noindent {\it Background response tensor.}~In light of the \emph{linear} mapping $\bq\mapsto \ubold_b$, one may define the second-order tensor $\bold{W}_b(\cdot,\bd)$ such that $\ubold_b(\bxi)=\bold{W}_b(\bxi,\bd)\bq$ for all $\bxi \in \R^3$. 

\noindent {\it Far-field pattern.}~The scattered waveform $\ubold^{sc} = \ubold^p \oplus \ubold^s$ in (\ref{ups}), affiliated with the damaged configuration in  Fig.~\ref{fig_elastic1}~(b),~possesses the following asymptotic representation at the far-field~\cite{CC06, CK}:
\begin{equation}\label{ffp}
\ubold^{sc}(\bxi) = \alpha_p\frac{e^{ik_pr}}{r}\ubold^{p,\infty}(\hat{\bxi}) + \alpha_s\frac{e^{ik_sr}}{r}\ubold^{s,\infty}(\hat{\bxi}) +O(r^{-2}), \qquad r = |\bxi| \rightarrow \infty,
\end{equation}
where $\ubold^{p,\infty}\!\!\parallel\! \hat\bxi$ and $\ubold^{s,\infty}\!\!\perp\! \hat\bxi$ are the so-called P- and S-wave patterns respectively as $r \rightarrow \infty$,
\begin{equation}\label{as}
\alpha_p=\frac{1}{4\pi(\lambda_0+2\mu_0)},~\text{and}~\alpha_s=\frac{1}{4\pi\mu_0}.
\end{equation}
In this setting, the elastic far-field pattern $\ubold^{\infty}$ is defined by
\begin{equation}\label{def-ff}
\ubold^{\infty}:=\ubold^{p,\infty}\oplus\ubold^{s,\infty}.
\end{equation}
%Moreover, the following expressions hold (correct!):
%\begin{eqnarray}
%\ubold^{p,\infty}(\xhat)&=& -ik_p\xhat\int_{\Gz}\Big\{\lambda[\ubold] + 2\mu (\bnu \cdot\xhat)[\ubold]\cdot\xhat\}\\
%\ubold^{s,\infty}(\xhat)&=& -ik_s\xhat\times\int_{\Gz}\{\mu([\ubold]\times\xhat)+\mu(\bnu\times\xhat)([\ubold]\cdot\xhat)\}.
%\end{eqnarray}
\noindent {\it Kupradze matrix in $\Omega_{ext}$.}~The fundamental displacement tensor $\GG(\cdot,\bx)$ associated to the exterior domain $\Omega_{ext}(\mu_0, \rho_0, \lambda_0)$ satisfies 
\begin{equation}\label{gre}
\Delta_{0}^*\GG(\bxi,\bx) + \rho_0 \exs \omega^2 \GG(\bxi,\bx) = \delta(\bxi-\bx) \, \Id, \qquad  \bxi,\bx \in \R^3, \bxi \neq \bx,
\end{equation}
where $\delta(\bxi-\bx)$ is the Dirac distribution centered at $\bx$. $\GG(\bxi,\bx)$ is also know as the Kupradze matrix -- describing the induced wave motion at $\bxi \in \R^3$ due to a unit time-harmonic point force applied at $\bx \in \R^3$, which may be recast as~\cite{kupradze}:
\begin{equation}\label{Gf}
\GG(\bxi,\bx)= \alpha_s \phi_{k_s}(\bxi,\bx) \exs \Id + \alpha_p \nabla_{\! \bxi}  \nabla_{\! \bxi} \exs [ \phi_{k_s} \! -  \phi_{k_p} ](\bxi,\bx), \quad \phi_{k}(\bxi,\bx)=\frac{e^{ik|\bxi-\bx|}}{|\bxi-\bx|}.
\end{equation}
As $|\bxi| \!\rightarrow\! \infty$, the asymptotic behavior of (\ref{Gf}) takes a similar form as (\ref{ffp}) so that the far-field pattern of the Kupradze matrix can be described as 
\begin{equation} \label{ggf}
\GG^{\infty}(\hat\bxi,\bx) ~=~  \GG^{p,\infty}(\hat\bxi,\bx) \oplus \GG^{s,\infty}(\hat\bxi,\bx), \qquad  \forall \hat\bxi \in \Sp^2, \bx \in \R^3,   
\vspace{-2 mm}
\end{equation}
where
\begin{equation} \nonumber
\GG^{p,\infty}(\hat\bxi,\bx) ~=~e^{-i k_p \hat\bxi \cdot\bx} \,\exs \hat\bxi \otimes \hat\bxi, \qquad
 \GG^{s,\infty}(\hat\bxi,\bx)~=~e^{-i k_s \hat\bxi\cdot\bx}(\Id-\hat\bxi \otimes \hat\bxi).
\end{equation}
\begin{remark}\label{recip0}
In light of \eqref{Wi} and \eqref{ggf}, one may observe the following relation
\begin{equation}\nonumber
\GG^\infty(\hat\bxi, \bx)~=~\bold{W}^i(\bx,-\hat\bxi), \qquad  \forall \hat\bxi \in \Sp^2, \bx \in \R^3.
\end{equation}
\end{remark}
%-----------------------------------------------------------------------------------------------------------------------------------------------------------------------------------------------------------
\subsection{The mixed reciprocity theorem}
%-----------------------------------------------------------------------------------------------------------------------------------------------------------------------------------------------------------

This section aims to establish a suitable reciprocity principle affiliated with the background domain i.e.~the intact heterogeneous composite shown in Fig.~\ref{fig_elastic1}~(a). To this end, let $\Grb(\bxi,\bx) \in H_{loc}^1(\R^3\setminus\nxs\{\bx\})^{3\times3}$ define the displacement Green's tensor satisfying
\vspace{-0.5 mm}
\begin{equation} \label{ch1:Gr}
\begin{aligned}
&\Delta^{\! *}_{\bxi} \exs \Grb (\bxi,\bx) \,+\, \rho(\bxi) \exs \omega^2 \Grb (\bxi,\bx) \,=\,  \delta(\bxi-\bx) \, \Id, \quad & \bxi \in \R^3\!\setminus\! \{\overline{\G \cup \G_1}\}, \bxi \neq \bx  \in \R^3, & \\*[0.6 mm]
&\dnu \Grb^- (\bxi,\bx) \,+\, \dnu \Grb^ +(\bxi,\bx) \,=\, \bold{0}, \,\, \Grb^-(\bxi,\bx) \,=\, \Grb^+(\bxi,\bx), \, & \bxi \in \G, \bxi \neq \bx  \in \R^3, & \\*[0.7 mm]
&\dnu \Grb^+ (\bxi,\bx) \,+\, \dnu \Grb^\circ(\bxi,\bx) \,=\, \bold{0}, \,\, \Grb^+(\bxi,\bx) \,=\, \Grb^\circ(\bxi,\bx),  \,& \bxi \in \G_1, \bxi \neq \bx  \in \R^3. &
\end{aligned}
\end{equation}
One should bear in mind that for every polarization vector $\bq\in\R^3$, the Green's function $\Grb(\bxi,\bx)\bq$ also satisfies the elastic radiation condition, similar to (\ref{kup-cond}). In this setting, the following mixed reciprocity condition applies. 
\begin{theorem}\label{theo-mrp}
The asymptotic expansion of \eqref{ch1:Gr} as $|\bxi| \rightarrow 0$ yields the Green's far-field pattern $\Grb^{\infty}(\cdot,\bx)$ that holds the following relation with the background response tensor $\bold{W}_b$.   
\begin{equation}\label{MRP}
\Grb^{\infty}(\hat\bxi,\bx)~=~\bold{W}_b(\bx,-\hat\bxi), \qquad  \forall \hat\bxi \in \Sp^2, \bx \in \R^3.
\end{equation}
\end{theorem}
\begin{proof}
The above statement will be analyzed in two steps considering~(1)~$\bx\in \Omega_{ext}$, and (2)~$\bx\in \Omega_{+} \cup \exs \Omega_{-}$. In { \it Step 1}, the source point $\bx$ is assumed to be located in the external domain $\Omega_{ext}$. In this setting, every column of $\Grb(\cdot,\bx)-\GG(\cdot,\bx)$ represents a regular radiating solutions of the Navier equation in $\Omega_{ext}$ satisfying the radiation conditions (\ref{kup-cond}). Thus, Betti's theorem reads $\forall \xbold,\bx \in \Omega_{ext}\backslash \Gamma_1$,
 \begin{equation} \label{ext1}
\begin{aligned}
  \big[\Grb-\GG\big](\bxi,\bx) &\,= \int_{\G_1}\Big\{ \dnu \GG(\bxi,\ybold) (\Grb-\GG)(\ybold,\bx) \,\,-  \\*[0.2mm]
  &\qquad\qquad\,\,\, - \GG(\bxi,\ybold) \dnu(\Grb-\GG)(\ybold,\bx)\Big\}\dsy \,= \\*[0.5mm]
  &= \int_{\G_1}\big\{   \dnu\GG(\bxi,\ybold) \Grb(\ybold,\bx)
  -\GG(\bxi,\ybold) \dnu \Grb(\ybold,\bx)\big\}\dsy, 
  \end{aligned}\vspace{-0.5mm}
 \end{equation} 
where the co-normal derivatives $\dnu(\cdot)$ of the Green's tensors are with respect to the first argument. The second identity in \eqref{ext1} makes use of
$$\int_{\G_1}\big\{ \partial_{\bnu}^*\GG(\xbold,\ybold) \GG(\ybold,\zbold)-\GG(\xbold,\ybold) \partial_{\bnu}^*\GG(\ybold,\zbold)\big\}\dsy=\bold{0}, \quad \forall \xbold,\bx \in \Omega_{ext}\backslash \Gamma_1.$$
\vspace{-1 mm}
In view of Remark~\ref{recip0}, the asymptotic behavior of \eqref{ext1} as $|\bxi| \rightarrow \infty$ may be recast as
\begin{equation} \label{ext*}
\begin{aligned}
\Grb^\infty(\xhat,\zbold)- \bold{W}^i(\zbold, -\xhat) &\,=  \int_{\G_1}\Big\{  \partial_{\bnu}^*\bold{W}^i(\ybold, -\xhat)\exs \Grb(\ybold,\zbold) \,-\\*[0.5 mm]
  &\,\,-\bold{W}^i(\ybold,-\xhat) \exs \partial_{\bnu}^*\Grb(\ybold,\zbold)\Big\}\dsy, \quad \forall \bx \in \Omega_{ext}\backslash \Gamma_1,\,\hat\bxi \in \Sp^2.
 \end{aligned}  \vspace{-2 mm}
 \end{equation}
In what follows, the right hand side of \eqref{ext*} is shown to represent the scattered wave tensor $\bold{W}_b^{sc}(\zbold, -\xhat) \colon\!\!\! = [\bold{W}_b - \bold{W}^i](\zbold, -\xhat)$ of the background domain, which proves the theorem statement for $\zbold \in \Omega_{ext}$. In this vein, the boundary integral representation of $\bold{W}_b^{sc}(\zbold,-\xhat)$ can be written in terms of the reciprocity relation between $\Grb(\ybold,\zbold)$ and $\bold{W}_{b}^{sc}(\ybold,-\xhat)$ over $\Omega_{ext}$, so that $\forall \bx \in \Omega_{ext}\backslash \Gamma_1,\,\hat\bxi \in \Sp^2$,
\begin{equation}
\label{ext2}
\bold{W}_b^{sc}(\zbold,-\xhat)= \int_{\G_1}\Big\{ \dnu\Grb(\ybold,\zbold)\bold{W}_{b}^{sc}(\ybold,-\xhat) \\*[1mm]
-\Grb(\ybold,\zbold)\dnu \bold{W}_{b}^{sc}(\ybold,-\xhat)\Big\}\dsy.
\end{equation}
On the other hand, the continuity of $\bold{W}_b(\cdot,-\hat\bxi)$ and $\Grb(\cdot,\zbold)$ across $\G$ and their reciprocity relation over $\Omega_{+}\cup\Omega_{-}$ result in
\begin{equation} \label{As}
\begin{aligned} 
& \boldsymbol{0} \,=  \int_{\Omega_+\cup \exs \Omega_-} \!\!\! \Big\{ \Delta^{\! *} \Grb(\ybold,\zbold)\bold{W}_b(\ybold,-\xhat) -
\Grb(\ybold,\zbold) \exs \Delta^{\! *}  \bold{W}_b(\ybold,-\xhat)\Big\}\dy \\*[1 mm]
& \quad\! = \int_{\G_1}\Big\{  \dnu\Grb(\ybold,\zbold)\bold{W}_{b}(\ybold,-\xhat) -
\Grb(\ybold,\zbold) \dnu \bold{W}_{b} (\ybold,-\xhat)\Big\}\dsy, \quad \forall \bx \in \Omega_{ext}\backslash \Gamma_1.
\end{aligned}
\end{equation}
On decomposing $\bold{W}_b(\ybold,-\xhat)$ in \eqref{As} into incident and scattered components, one finds:
\begin{equation}
\begin{aligned} \label{WiWs}
 &\int_{\G_1}\Big\{  \dnu\Grb(\ybold,\zbold)\bold{W}_{b}^{sc}(\ybold,-\xhat) -
\Grb(\ybold,\zbold) \dnu \bold{W}_{b}^{sc} (\ybold,-\xhat)\Big\}\dsy ~=~\\*[1 mm]
 &  \int_{\G_1}\Big\{  \dnu \bold{W}^{i} (\ybold,-\xhat) \Grb(\ybold,\zbold)  -
                \bold{W}^{i}(\ybold,-\xhat)  \dnu\Grb(\ybold,\zbold) \Big\}\dsy.
\end{aligned}
\end{equation}
%\begin{equation}\nonumber
%\begin{aligned}
%&\int_{\G_1}\Big\{ (\GG-\Grb)(\ybold,\zbold) \partial^*_{\bnu} \bold{W}^{sc}_{b}(\ybold,-\xhat)\,- \\*[1mm]
%& \qquad\,\, -\partial^*_{\bnu} (\GG-\Grb)(\ybold,\zbold)\bold{W}^{sc}_{b}(\ybold,-\xhat)\Big\}\dsy=\bold{0}, \quad \forall \bx \in \Omega_{ext}\backslash \Gamma_1,\,\hat\bxi \in \Sp^2.
% \end{aligned}
%\end{equation}
%but by the integral representation of $\bold{W}_b^{sc}(\cdot,-\xhat)$  in $\Omega_{ext}$ (see \cite{kupradze}), we have that
%
%Additionally, from the transmission conditions of the background problem,
% \begin{eqnarray}
%\label{ext4}
%&& \int_{\G_1}\Big\{  \dnuy\Grb(\zbold,\ybold)\bold{W}_{b,ext}(\ybold,-\xhat)-
%\Grb(\zbold,\ybold) \dnuy \bold{W}_{b,ext} (\ybold,-\xhat)\Big\}\dsy\nonumber\\*[1mm]
%&=&\int_{\G_1}\Big\{  \dnuy\Grb (\zbold,\ybold)\bold{W}_{b,+}(\ybold,-\xhat)-
%\Gr_{b,+}(\zbold,\ybold) \dnuy \bold{W}_{b,+}(\ybold,-\xhat)\Big\}\dsy\nonumber\\*[1mm]
%&=&\int_{\Omega_+\cup \Omega_-}\Big\{ \Delta^*_{y} \Grb(\zbold,\ybold)\bold{W}_b(\ybold,-\xhat) -
%\Grb(\zbold,\ybold) \Delta^*_y  \bold{W}_b(\ybold,-\xhat)\Big\}\dy \nonumber\\*[1mm]
%&+&\int_{\G}\Big\{ \dnuy\Grb(\zbold,\ybold)\left[\bold{W}_b\right](\ybold,-\xhat)-
%\Grb(\zbold,\ybold) \left[\dnuy \bold{W}_b\right](\ybold,-\xhat)\Big\}\dsy\nonumber\\*[1mm]
%&=&\bold{0}.
% \end{eqnarray}
On the basis of (\ref{ext2}) and (\ref{WiWs}), one concludes
\begin{equation}
\label{zext}
 \bold{W}_b^{sc}(\zbold,-\xhat)=\int_{\G_1}\Big\{  \dnu \bold{W}^{i} (\ybold,-\xhat) \Grb(\ybold,\zbold)  -
                \bold{W}^{i}(\ybold,-\xhat)  \dnu\Grb(\ybold,\zbold) \Big\}\dsy.
\end{equation}
This, in light of \eqref{ext*}, completes the proof for $\zbold \in \Omega_{ext}$ as the following
\begin{eqnarray}\label{s1}
\Grb^\infty(\xhat,\zbold)= \big[\bold{W}^i+  \bold{W}_b^{sc}\big](\zbold,-\xhat) =  \bold{W}_b(\zbold,-\xhat), \qquad \forall \bx \in \Omega_{ext}\backslash \Gamma_1,\,\hat\bxi \in \Sp^2.
\end{eqnarray}
Now, let us move on to \emph{Step 2} where $\zbold\in \Omega_{+}\nxs\cup\Omega_-$ and $\Grb(\cdot,\zbold)$ is a smooth radiating solution of \eqref{ch1:Gr} in $\Omega_{ext}$. In this setting, the reciprocity relation between $\GG(\cdot,\ybold)$ and $\Grb(\cdot,\zbold) $ over $\Omega_{ext}$ reads that $\forall \zbold \in \lbrace\Omega_{+}\nxs\cup\Omega_-\rbrace \backslash \Gamma_1, \, \xbold\in\Omega_{ext}$,
\begin{equation}
 \label{int1}
  \Grb(\xbold,\zbold) = \int_{\G_1}\Big\{\dnu\GG(\xbold,\ybold)\Grb(\ybold,\zbold) 
  \nonumber\\*[1mm]
-\GG(\xbold,\ybold) \dnu\Grb(\ybold,\zbold)\Big\}\dsy,  
 \end{equation}
whose far-field expansion as $|\bxi| \rightarrow \infty$ takes the form 
 \begin{equation}
 \label{int2}
  \Grb^\infty (\xhat,\zbold) =  \int_{\G_1}\Big\{\dnu\GG^\infty(\xhat,\ybold) \Grb(\ybold,\zbold) -
\GG^\infty(\xhat,\ybold) \dnu\Grb(\ybold,\zbold)\Big\}\dsy, \quad \hat\bxi \in \Sp^2.
 \end{equation}
In addition, one may note that $\bold{W}^{sc}_b(\cdot, -\xhat)$ is also a radiating solution to the Navier equation whose relation with $\Grb(\ybold,\zbold)$ can be described in terms of Betti's formula over $\Omega_{ext}$, so that $\forall \zbold \in \lbrace\Omega_{+}\nxs\cup\Omega_-\rbrace \backslash \Gamma_1, \, \hat\bxi \in \Sp^2$,
 \begin{equation}
 \label{int3}
  \bold{0} ~=\,  \int_{\G_1}\Big\{ \dnu \bold{W}_{b}^{sc}(\ybold,-\xhat) \Grb(\ybold,\zbold) - \bold{W}_{b}^{sc}(\ybold,-\xhat) \dnu \Grb(\ybold,\zbold) \Big\}\dsy.
 \end{equation}
 Invoking Remark~\ref{recip0}, one may superimpose~\eqref{int2} and~\eqref{int3} to obtain $\forall \zbold \in \lbrace\Omega_{+}\nxs\cup\Omega_-\rbrace \backslash \Gamma_1$ and $\hat\bxi \in \Sp^2$,
 \begin{equation}\label{eq+}
 \Grb^\infty(\xhat,\zbold) =  \int_{\G_1}\Big\{\dnu \bold{W}_b(\ybold,-\xhat)  \Grb(\ybold,\zbold) 
 -\bold{W}_b(\ybold,-\xhat) \dnu \Grb(\ybold,\zbold)\Big\}\dsy.
 \end{equation}
 On the other hand, by considering the continuity of $\bold{W}_b(\ybold,-\xhat)$ and $ \Grb(\ybold,\zbold)$ across $\G$, one may integrate by part the following weak form over $\Omega_{+}\nxs\cup\Omega_-$,
 \begin{equation}\nonumber
\bold{W}_b(\zbold,-\xhat) ~=\,\int_{\Omega_+\nxs\cup\Omega_-}\!\!\Big\{  \Delta^*_{\bnu} \bold{W}_b(\ybold,-\xhat) \Grb(\ybold,\zbold)
- \bold{W}_b (\ybold,-\xhat) \Delta^*_{\bnu} \Grb(\ybold,\zbold)\Big\}\dy, 
 \end{equation}
 to obtain $\forall \hat\bxi \in \Sp^2$ and $\zbold \in \lbrace\Omega_{+}\nxs\cup\Omega_-\rbrace \backslash \Gamma_1$
 \begin{equation}\label{eq++}
 \bold{W}_b(\zbold,-\xhat) ~=\, \int_{\G_1}\Big\{\dnu \bold{W}_b(\ybold,-\xhat)  \Grb(\ybold,\zbold) 
 -\bold{W}_b(\ybold,-\xhat) \dnu \Grb(\ybold,\zbold)\Big\}\dsy.
 \end{equation}
 In view of~\eqref{eq+} and \eqref{eq++}, one concludes
 \begin{eqnarray}\label{s2}
\Grb^\infty(\xhat,\zbold)~=~\bold{W}_b(\zbold,-\xhat), \qquad \forall \bx \in\lbrace\Omega_{+}\nxs\cup\Omega_-\rbrace \backslash \Gamma_1,\,\hat\bxi \in \Sp^2.
\end{eqnarray}
The proof of mixed reciprocity relation $\Grb^\infty(\xhat,\zbold)=\bold{W}_b(\zbold, -\xhat)$ is now complete everywhere in $\zbold \in \R^3$ based on~\eqref{s1},~\eqref{s2} and the continuity of $\Grb$ and $\bold{W}_b$ over $\G_1$.
\end{proof}
%-----------------------------------------------------------------------------------------------------------------------------------------------------------------------------------------------------------
\subsection{Far-field and scattering operators}\label{sec33}
%-----------------------------------------------------------------------------------------------------------------------------------------------------------------------------------------------------------

\noindent {\it Multiple-incident scattering.}~The interaction of $\ubold_\g \in H^1_{loc}(\R^3)^3$ with $\Omega$ in the damaged condition gives birth to the scattered field $\us \in H^1(\R^3\setminus\overline{\G}_0)^3$ satisfying the Kupradze radiation condition and 
\vspace{-1 mm}
\begin{equation} \label{ch1:Gr2}
\begin{aligned}
&\Delta^{\! *} \exs \us(\bxi) \,+\, \rho(\bxi) \exs \omega^2 \us(\bxi) \,=\,   \bold{0}, \quad & \bxi \in \R^3\backslash \{\overline{\G \cup \G_1}\}, \\*[0.5 mm]
&\dnu \us^- (\bxi) \,+\, \dnu \us^+(\bxi) \,=\, \bold{0}, \,\,  \, & \bxi \in \G,  \\*[0.5 mm]
&\dnu \us^- (\bxi) \,=\, \KK(\bxi)[\us](\bxi) - \dnu \ubold_{\g}^- (\bxi)  \,& \bxi \in \G_0,  \\*[0.5 mm]
&\us^- (\bxi) \,=\, \us^+(\bxi),& \bxi \in \G\backslash\overline{\G_0},   \\*[0.5 mm]
&\dnu \us^+(\bxi) \,+\, \dnu \us^{\circ}(\bxi) \,=\, \bold{0}, \,\, \us^+(\bxi) \,=\, \us^{\circ}(\bxi),  \,& \bxi \in \G_1. 
\end{aligned}
\end{equation}
where $\lbrace  \circ, +, - \rbrace$ indicate affiliation with $\lbrace  \Omega_{ext}, \Omega_+, \Omega_- \rbrace$ respectively.

\vspace{1 mm}
\noindent {\it Far-field operator.}~In this setting, let $\Fop:L^2(\Sp^2)^3\rightarrow L^2(\Sp^2)^3$ denote the far-field operator associated with the damaged configuration, Fig.~\ref{fig_elastic1}~(b), defined as the following
$$\Fop(\g)(\hat\bxi)~=~\us^\infty(\hat\bxi), \qquad \forall \g \in L^2(\Sp^2)^3,\,\, \hat\bxi \in \Sp^2,$$ 
where $\us^\infty$ is the far field pattern associated with $\us \in H^1(\R^3\setminus\overline{\G}_0)^3$ in~\eqref{ch1:Gr2}, according to asymptotic representation~\eqref{ffp}--\eqref{def-ff}.
Similarly, one may define the background far-field operator  $\Fop_b:L^2(\Sp^2)^3\rightarrow L^2(\Sp^2)^3$ 
$$\Fop_b(\g)(\hat\bxi)~=~\ubold^\infty_{b,\g}(\hat\bxi),  \qquad \forall \g \in L^2(\Sp^2)^3,\,\, \hat\bxi \in \Sp^2,$$
where $\ubold_{b,\g}^\infty$ is the scattered far field pattern (as $|\bxi| \rightarrow \infty$) associated with
\begin{equation}\label{bgsc}
 \ubold_{b,\g}^{sc}(\bxi) ~=\, \int_{\Sp^2} \big[\boldsymbol{W}_b - \boldsymbol{W}^i\big](\bxi,\bd) \exs \g(\bd) \,\, \text{d}S_{\bd}, \quad  \forall \bxi \in \R^3,
\end{equation}
which is defined over the background domain. Based on this, one may identify the far-field signature of damage by
\begin{equation}\label{Fop}
 \Fop_D(\g)(\hat\bxi):=[\Fop-\Fop_b](\g)(\hat\bxi), \qquad \forall \g \in L^2(\Sp^2)^3,\,\, \hat\bxi \in \Sp^2. 
 \end{equation}

\noindent {\it Scattering operator.}~The scattering operator ${\mathcal S}_b:L^2(\Sp^2)^3\rightarrow L^2(\Sp^2)^3$
is defined by 
\begin{equation}\label{S_b}
{\mathcal S}_b~=~I \exs-\exs 2i \exs ( k_p \alpha_p \oplus k_s \alpha_s) \! \cdot \! \Fop_b,
\end{equation}
where $I$ denotes the identity map, and $(\alpha_p, \alpha_s)$ is given in \eqref{as}.

\begin{remark}\label{SSI}
Using a similar argument to that of Theorem~1.8 in~\cite{kirsch}, one may show that the scattering operator is unitary, that is ${\mathcal S}_b{\mathcal S}_b^*={\mathcal S}_b^*{\mathcal S}_b=I$. 
\end{remark}

\begin{theorem}\label{propSb}
With reference to \eqref{s2} and \eqref{S_b}, the following identity holds
\begin{equation}
\bold{W}_b(\zbold,-\xhat)={\mathcal S}_b( \overline{\bold{W}_b(\zbold,\cdot)})(\xhat), \qquad \forall\zbold\in\Omega, \, \xhat\in\Sp^2.
\end{equation}
\end{theorem}

\begin{proof}
For $\xbold\in\Omega_{ext}$ and $\zbold\in\Omega$, let the open ball $B_R \subset \R^3$ (of radius $R>0$) contain $\{\xbold\}\cup\overline{\Omega}\subset B_R$, so that $(\Grb-\overline{\Grb})(\cdot,\zbold)$ and $(\Grb-\GG)(\cdot,\xbold)$ are two regular radiating solutions solutions of the Navier equation in the background domain according to \eqref{ch1:Gr} and \eqref{Gf}. In this setting, the Betti's reciprocity identity between $\GG(\cdot,\xbold)$ and $(\Grb-\overline{\Grb})(\cdot,\zbold)$ over $\Omega_R:=\Omega_{ext}\cap B_R$ reads: 
\begin{equation}\label{id1}
\begin{aligned}
(  \Grb-\overline{\Grb}  )(\xbold,\zbold) 
 =\int_{  \G_1 \cup S_R} \!\! & \big\{\dnu\GG(\ybold,\xbold)(\Grb-\overline{\Grb})(\ybold,\zbold) \\*[1mm]
&\!\!\!- \GG(\ybold,\xbold)\dnu(\Grb-\overline{\Grb})(\ybold,\zbold) \big\}\dsy, \quad  \forall\xbold\in\Omega_{ext}, \, \zbold\in\Omega.
\end{aligned}
\end{equation}
Similarly, the reciprocity relation between $(\Grb-\GG)(\cdot,\zbold)$ and $(\Grb-\overline{\Grb})(\cdot,\zbold)$ over $\Omega_R$ satisfy
\begin{equation}\label{id2}
\begin{aligned}
{0}~=\,\int_{  \G_1 \cup S_R} \!\! & \big\{\dnu(\Grb-\GG)(\ybold,\xbold)(\Grb-\overline{\Grb})(\ybold,\zbold) \\*[1mm]
&\!\!\!- (\Grb-\GG)(\ybold,\xbold)\dnu(\Grb-\overline{\Grb})(\ybold,\zbold) \big\}\dsy, \quad  \forall\xbold\in\Omega_{ext}, \, \zbold\in\Omega.
\end{aligned}
\end{equation}
Combining (\ref{id1}) and (\ref{id2}) results in 
\begin{equation}\label{id3}
\begin{aligned}
(  \Grb-\overline{\Grb}  )(\xbold,\zbold) 
 =\int_{  \G_1 \cup S_R} \!\! & \big\{\dnu\Grb(\ybold,\xbold)(\Grb-\overline{\Grb})(\ybold,\zbold) \\*[1mm]
&\!\!\!\!\!- \Grb(\ybold,\xbold)\dnu(\Grb-\overline{\Grb})(\ybold,\zbold) \big\}\dsy, \quad  \forall\xbold\in\Omega_{ext}, \, \zbold\in\Omega.
\end{aligned}
\end{equation}
Applying Betti's theorem to $\Grb(\cdot,\xbold)$ and $(\Grb-\overline{\Grb})(\cdot,\zbold)$ over $\Omega$, one may observe that the integral over $\G_1$ in (\ref{id3}) vanishes i.e.,~$\forall\xbold\in\Omega_{ext}, \, \zbold\in\Omega$,
\begin{equation}\label{id4}
\begin{aligned}
0~=\,\int_{  \G_1} \!\! & \big\{\dnu\Grb(\ybold,\xbold)(\Grb-\overline{\Grb})(\ybold,\zbold) 
&\!\!\!- \Grb(\ybold,\xbold)\dnu(\Grb-\overline{\Grb})(\ybold,\zbold) \big\}\dsy. %\quad  .
\end{aligned}
\end{equation}
On the other hand, the reciprocity relation between $\Grb(\cdot,\xbold)$ and $\Grb(\cdot,\zbold)$ in $B_R$ reads
\begin{equation}\label{id5}
\begin{aligned}
0~=\,\int_{  S_R} \!\! & \big\{\dnu\Grb(\ybold,\xbold)\Grb(\ybold,\zbold) 
&\!\!\!- \Grb(\ybold,\xbold)\dnu\Grb(\ybold,\zbold) \big\}\dsy. %\quad  \forall\xbold\in\Omega_{ext}, \, \zbold\in\Omega.
\end{aligned}
\end{equation}
In view of \eqref{id4} and \eqref{id5}, \eqref{id3} may be recast as, $\forall\xbold\in\Omega_{ext}, \, \zbold\in\Omega$,
\begin{equation}\label{id6}
\begin{aligned}
(  \Grb-\overline{\Grb}  )(\xbold,\zbold) 
 = - \int_{  S_R} \!\! & \big\{\dnu\Grb(\ybold,\xbold)\overline{\Grb}(\ybold,\zbold) 
&\!\!\!- \Grb(\ybold,\xbold)\dnu\overline{\Grb}(\ybold,\zbold) \big\}\dsy. % \quad  \forall\xbold\in\Omega_{ext}, \, \zbold\in\Omega.
\end{aligned}
\end{equation}
When the radius of support $S_R$ in \eqref{id6} is taken to infinity, the following asymptotic representations hold for ${\Grb}(\ybold,\cdot)$, $\forall \hat\ybold \in \Sp^2$,  
\begin{equation}\label{id7}
\begin{aligned}
&\Grb(\ybold,\cdot) \,=\,  \alpha_p\frac{e^{ik_pr}}{r}\Grb^{p,\infty}(\hat{\ybold},\cdot) + \alpha_s\frac{e^{ik_sr}}{r}\Grb^{s,\infty}(\hat{\ybold},\cdot) +O(r^{-2}), \quad r = |\ybold| \rightarrow \infty,\\*[1 mm]
&\dnu\Grb(\ybold,\cdot) \,=\,  i k_p \frac{e^{ik_pr}}{4 \pi r} \Grb^{p,\infty}(\hat{\ybold},\cdot) + i k_s \frac{e^{ik_sr}}{4 \pi r}\Grb^{s,\infty}(\hat{\ybold},\cdot) +O(r^{-2}), \quad r = |\ybold| \rightarrow \infty.
\end{aligned}
\end{equation}
Substituting \eqref{id7} into \eqref{id6}, the first integral takes the following form
\begin{equation}\label{id8}
\begin{aligned}
 \int_{  S_R} \!\! \dnu\Grb&(\ybold,\xbold)\overline{\Grb}(\ybold,\zbold) \dsy ~=\,  i (k_p \alpha_p \oplus k_s \alpha_s) \times \\*[0.5 mm]
& \int_{\Sp^2} \big[    \Grb^{p,\infty}(\hat{\ybold},\xbold)  \overline{\Grb^{p,\infty}}(\hat{\ybold},\zbold) \oplus \Grb^{s,\infty}(\hat{\ybold},\xbold)  \overline{\Grb^{s,\infty}}(\hat{\ybold},\zbold) \big]  \,\, \text{d}S_{\hat\ybold} \,=\\*[0.75 mm]
&\quad\!\!\! =~  i (k_p \alpha_p \oplus k_s \alpha_s)  \int_{\Sp^2} \Grb^{\infty}(\hat{\ybold},\xbold)  \overline{\Grb^{\infty}}(\hat{\ybold},\zbold) \,\, \text{d}S_{\hat\ybold},\qquad |\ybold| \rightarrow \infty.
\end{aligned}
\end{equation}
where $\Grb^{\infty} = \Grb^{p,\infty} \oplus \Grb^{s,\infty}$. As a result, \eqref{id6} can be written as
\begin{equation}\label{id9}
\begin{aligned}
(  \Grb-\overline{\Grb}  )(\xbold,\zbold) \,=-2 i \exs (k_p \alpha_p \oplus k_s \alpha_s) \! \int_{\Sp^2} \Grb^{\infty}(\hat{\ybold},\xbold)  \overline{\Grb^{\infty}}(\hat{\ybold},\zbold) \, \text{d}S_{\hat\ybold},\quad |\ybold| \rightarrow \infty.\end{aligned}
\end{equation}
In parallel, as $|\xbold| \rightarrow \infty$, one may find a proper boundary integral representation for $\Grb^\infty(\xhat,\zbold)-\overline{\Grb^\infty}(-\xhat,\zbold)$ in view of \eqref{int2} and the identity $\overline{\GG^\infty}(-\xhat,\ybold)=\GG^\infty(\xhat,\ybold)$, so that
 \begin{equation}
 \label{int2bar}
  \overline{\Grb^\infty} (-\xhat,\zbold) =  \int_{\G_1}\Big\{\dnu\GG^\infty(\xhat,\ybold) \overline{\Grb}(\ybold,\zbold) -
\GG^\infty(\xhat,\ybold) \dnu\overline{\Grb}(\ybold,\zbold)\Big\}\dsy, \quad \hat\bxi \in \Sp^2.
 \end{equation}
Now, in light of the mixed reciprocity principle \eqref{MRP}, $$\Grb^\infty(\xhat,\zbold)-\overline{\Grb^\infty}(-\xhat,\zbold) = \bold{W}_b(\zbold,-\xhat) -  \overline{\bold{W}_b}(\zbold,\xhat),$$ one may write
\begin{equation}\label{id10}
\begin{aligned}
 \bold{W}_b(\zbold,-\xhat) -  \overline{\bold{W}_b}(\zbold,\xhat)
 =&\int_{  \G_1}   \big\{\dnu\GG^\infty(\hat\xbold,\ybold)(\Grb-\overline{\Grb})(\ybold,\zbold) \\*[1mm]
&\!\!\!- \GG^\infty(\hat\xbold,\ybold)\dnu(\Grb-\overline{\Grb})(\ybold,\zbold) \big\}\dsy, \quad \hat\xbold \in \Sp^2, \bx \in \Omega. 
\end{aligned}
\end{equation}
Substituting \eqref{id9} into \eqref{id10}, the first integral can be rewritten as
\begin{equation}\label{id11}
\begin{aligned}
\int_{  \G_1}  & \dnu\GG^\infty(\hat\xbold,\ybold)(\Grb-\overline{\Grb})(\ybold,\zbold) \dsy = -2 i \exs (k_p \alpha_p \oplus k_s \alpha_s) \times   \\*[0.5mm] 
& \int_{\Sp^2} \overline{\Grb^{\infty}}(\bd,\zbold) \int_{  \G_1}  \dnu\GG^\infty(\hat\xbold,\ybold) \Grb^{\infty}(\bd,\ybold)   \dsy \text{d}S_{\bd} , \quad \hat\xbold \in \Sp^2, \bx \in \Omega. 
\end{aligned}
\end{equation}
As a result, \eqref{id10} is recast as the following
\begin{equation}\label{id12}
\begin{aligned}
& \bold{W}_b(\zbold,-\xhat) -  \overline{\bold{W}_b}(\zbold,\xhat)
 =  -2 i \exs (k_p \alpha_p \oplus k_s \alpha_s) \int_{\Sp^2} \overline{\Grb^{\infty}}(\bd,\zbold) \times   \\*[0.5mm] 
&  \int_{  \G_1} \Big\lbrace \dnu\GG^\infty(\hat\xbold,\ybold) \Grb^{\infty}(\bd,\ybold)  - \GG^\infty(\hat\xbold,\ybold)  \dnu\Grb^{\infty}(\bd,\ybold) \Big\rbrace  \dsy \text{d}S_{\bd} , \quad \hat\xbold \in \Sp^2, \bx \in \Omega, 
\end{aligned}
\end{equation}
which may be further simplified according to \eqref{int2} and Lemma 4.1 in~\cite{fatemeh} as
\begin{equation}\label{id13}
\begin{aligned}
 \bold{W}_b(\zbold,-\xhat) -  \overline{\bold{W}_b}(\zbold,\xhat)
 &=  -2 i \exs (k_p \alpha_p \oplus k_s \alpha_s) \int_{\Sp^2} \overline{\bold{W}_b}(\zbold,-\bd) \bold{W}_b^\infty(\bd,-\hat\bxi) \text{d}S_{\bd}\\*[0.2 mm]
&  =  -2 i \exs (k_p \alpha_p \oplus k_s \alpha_s) \int_{\Sp^2} \bold{W}_b^\infty(\hat\bxi,-\bd) \overline{\bold{W}_b}(\zbold,-\bd)  \text{d}S_{\bd}\\*[0.2 mm]
&  =  -2 i \exs (k_p \alpha_p \oplus k_s \alpha_s) \int_{\Sp^2} \bold{W}_b^\infty(\hat\bxi,\bd) \overline{\bold{W}_b}(\zbold,\bd)  \text{d}S_{\bd}\\*[0.2 mm]
&  =  -2 i \exs (k_p \alpha_p \oplus k_s \alpha_s) \Fop_b(\overline{\bold{W}_b}(\zbold,\cdot))(\hat\bxi), \qquad \hat\xbold \in \Sp^2, \bx \in \Omega. 
\end{aligned}
\end{equation}
\end{proof}

%---------------------------------------------------------------------------------------------------------------------------------------------------------------------------
\subsection{The factorization method}
%---------------------------------------------------------------------------------------------------------------------------------------------------------------------------

Within the framework of factorization method (FM), this section aims to characterize the support of interlayer delaminations $\Gz$, shown in Fig.~\ref{fig_elastic1}~(b), in terms of the range of operator $\Fop_D$, defined in \eqref{Fop}. To this end, the fundamental elements of FM for imaging in elastic composites, by way of ultrasonic waves, are established in Lemmas~\ref{lemmaH}~--~\ref{ch2lemmaG}. Such results pave the way for the main Theorem~\ref{theo-FM} furnishing a solid platform for non-iterative reconstruction of elastic fractures propagating along bimaterial interfaces. In what follows, for a generic Hilbert space $H$ and a bounded linear operator $F :H\rightarrow H$, we define the real and imaginary parts of operator $F$ by
\begin{equation}
\Re(F):=\frac{F + F^*}{2}\quad\text{and}\quad \Im(F):=\frac{F-F^*}{2i}.
\end{equation}    

\noindent Let us also define the Hergoltz operator $\Hop:L^2(\Sp^2)^3\rightarrow H^{-1/2}(\Gz)^3$ such that 
\begin{equation}
\Hop(\g)\!\nxs~\colon\!\!\!=~\!\!\nu(\bxi) \nxs\cdot\nxs \textrm{\bf C}(\bxi) \colon\!\nxs \nabla  \ubold_{b, \g} (\bxi), \quad  \forall \bxi \in \Gz,
\end{equation}
\vspace{-1 mm}
where 
\vspace{-1 mm}
$$ \ubold_{b,\g}(\bxi) =\! \int_{\Sp^2} \boldsymbol{W}_b(\bxi,\bd) \exs \g(\bd) \,\, \text{d}S_{\bd}, \quad  \forall \bxi \in \R^3. $$

\begin{assumption}\label{ass1}
\noindent Here, it is assumed that $\Gz$ and $\omega$ are such that $\Hop$ is injective -- i.e.~there are no non-trivial background fields $ \ubold_{b,\g}$ satisfying $\dnu  \ubold_{b,\g}|_{\Gz}=\bold{0}$. For more clarity, the latter premise implies $\ubold_{b,\g}=\bold{0}$ whenever $\Hop\g=\bold{0}$. This, by well-posedness of the background scattering problem~\eqref{pr1bb}, reads $\ubold_\g=\bold{0}$, and therefore, $\g=\bold{0}$, which establishes the injectivity of Herglotz operator $\Hop$.  
\end{assumption} 

\begin{remark}
 It is worth mentioning that according to~\cite[Lemma 5.3]{fatemeh}, Assumption~\ref{ass1} holds $\forall \omega>0$ except for a discrete set of values without finite accumulation points. 
\end{remark}

\begin{lemma}\label{lemmaH}
Under Assumption~\ref{ass1}, the Herglotz operator $\Hop$ has a dense range and its adjoint i.e.~conjugate transpose operator $\Hop^*:\widetilde{H}^{1/2}(\Gz)^3\rightarrow L^2(\Sp^2)^3$ satisfies:
\begin{equation}\label{H*}
{\mathcal S}_b\Hop^*\boldsymbol{\eta} = \int_{\Gz}\dnu \Gr^\infty_b(\cdot,\ybold)\boldsymbol{\eta}(\ybold)\dsy, \quad  \forall\boldsymbol{\eta}\in\widetilde{H}^{1/2}(\Gz)^3.
\end{equation}
\end{lemma}
\begin{proof} Let's define the fracture opening displacement profile $\boldsymbol{\eta}\in\widetilde{H}^{1/2}(\Gz)^3$, then in light of Theorem~\ref{theo-mrp}, one has
\begin{eqnarray}
\langle \Hop\g,\,\boldsymbol{\eta} \rangle_{\Gz} &=& \int_{\G_0}\dnu \ubold_{b,\g}(\ybold)\cdot \overline{\boldsymbol{\eta}}(\ybold) \dsy \nonumber\\*[1mm]
&=& \int_{\G_0}\overline{\boldsymbol{\eta}}(\ybold) \nxs\cdot\nxs \int_{\Sp^2}\dnu  \bold{W}_{b}(\ybold,\bd)\g(\bd)\, \text{d}S_{\bd} \! \dsy \nonumber\\*[1mm]
&=& \int_{\Sp^2}\g(\bd)\cdot\int_{\G_0}\dnu \Grb^\infty(-\bd,\ybold)\overline{\boldsymbol{\eta}}(\ybold) \! \dsy \text{d}S_{\bd}  \nonumber\\*[1mm]
&=& (\g,\Hop^*\boldsymbol{\eta})_{L^2(\Sp^2)^3},
\end{eqnarray}
where $\langle \cdot , \cdot \rangle_{\Gz}$ is understood in the sense of duality product $\langle{H}^{-1/2}(\Gz)^3 , \widetilde{H}^{1/2}(\Gz)^3 \rangle_{\Gz}$ as an extension to the $L^2$ inner product. As a result, 
\begin{equation}\label{eqH*bis}
(\Hop^*\boldsymbol{\eta})(\bd) = \int_{\G_0}\overline{\dnu \Grb^\infty(-\bd,\ybold)}\boldsymbol{\eta}(\ybold) \dsy, \qquad \g \in L^2(\Sp^2)^3.
\end{equation}
Lemma's statement (\ref{H*}) will then immediately follow from Theorem~\ref{propSb}. Next, the denseness of the range of $\Hop$ is derived from the injectivity of its adjoint operator $\Hop^*$ as the following. From (\ref{eqH*bis}), 
$$(\Hop^*\boldsymbol{\eta})(\bd)=\overline{\bv^\infty(-\bd)}, \qquad \forall\bd \in \Sp^2$$
 where $\bv^\infty$ is the far field pattern of the double-layer potential
 $$\bv= \int_{\G_0}\dnu\Grb(\cdot,\ybold)\overline{\boldsymbol{\eta}(\ybold)}\,ds(\ybold), \qquad \bv\in H^1(\R^3\setminus\overline{\Gz})^3,$$
satisfying
\vspace{0 mm}
\begin{equation} \label{ch3:Gr2}
\begin{aligned}
&\Delta^{\! *} \exs \bv(\bxi) \,+\, \rho(\bxi) \exs \omega^2 \bv(\bxi) \,=\,   \bold{0}, \quad & \bxi \in \R^3\backslash \{\overline{\G \cup \G_1}\}, \\*[0.5 mm]
&\dnu \bv^- (\bxi) \,+\, \dnu \bv^+(\bxi) \,=\, \bold{0}, \,\,  \, & \bxi \in \G,  \\*[0.5 mm]
&[\bv](\bxi) \,=\, [\bv^+ \nxs-\exs \bv^-](\bxi) \,=\,  \overline{\boldsymbol{\eta}}(\bxi), \,& \bxi \in \G_0,  \\*[0.5 mm]
&\bv^- (\bxi) \,=\, \bv^+(\bxi),& \bxi \in \G\backslash\overline{\G_0},   \\*[0.5 mm]
&\dnu \bv^+(\bxi) \,+\, \dnu \bv^{\circ}(\bxi) \,=\, \bold{0}, \,\, \bv^+(\bxi) \,=\, \bv^{\circ}(\bxi),  \,& \bxi \in \G_1, 
\end{aligned}
\end{equation}
complemented by the Kupradze radiation condition at infinity. Thus, if $\Hop^*\boldsymbol{\eta}=\bold{0}$, then $\bv^\infty=\bold{0}$ over the unit sphere of observation angles, and therefore, by Rellich's lemma~\cite[Lemma 2.11]{CK} $\bv(\bxi)=\bold{0}$ in $\bxi \in \Omega_{ext}$. Thanks to the Holmgren's theorem, the continuity of displacement and traction across $\G_1$, in \eqref{ch3:Gr2}, implies that $\bv=\bold{0}$ in an open neighborhood of $\G_1$. Then, the unique continuation principle reads $\bv(\bxi)=\bold{0}$ in $\bxi \in \Omega_+$. On repeating the application of Holmgren's theorem, one finds $\bv=\bold{0}$ in an open neighborhood of $\G\!\setminus\!\overline{\G_0}$, and hence by the unique continuation principle, $\bv=\bold{0}$ in $\Omega_-$. Therefore, $[\bv]=\bold{0}$ and then $\boldeta=\bold{0}$, which finishes the proof.
\end{proof}

\noindent In view of the linear slip interfacial condition over $\Gz$, let us define $T:H^{-1/2}(\Gz)^3\rightarrow\widetilde{H}^{1/2}(\Gz)^3$ by 
\begin{equation}\label{DT}
T(\tbold_{\g})(\bxi)~\colon \!\!\!\nxs=~[\us](\bxi), \quad  \tbold_{\g}(\bxi)~\colon \!\!\!\nxs=~\dnu \ubold_{b,\g}^- (\bxi), \qquad \bxi \in \Gz.
\end{equation}
Based on this, the following factorization is obtained for the far-field scattering operator $\Fop_D:L^2(\Sp^2)^3 \rightarrow L^2(\Sp^2)^3$,
\begin{equation}\label{FFD}
\Fop_D~=~{\mathcal S}_b\Hop^*T\Hop.
\end{equation}

\begin{lemma}\label{lemmaT}
$T$ is a bounded linear operator with decomposition $T=T_0+T_c$ where $T_0,T_c:H^{-1/2}(\Gz)^3\rightarrow\widetilde{H}^{1/2}(\Gz)^3$ are two bounded linear operators such that $T_0$ is coercive and self-adjoint, while $T_c$ is compact. As a result, $\Re(T)=T_0 + \Re(T_c)$, and  $ \Re(T_c)$ is compact.
\end{lemma}

\begin{proof}
The boundedness of operator $T$ is an immediate consequence of the well-posedness of (multiple-incident) scattering problem~\eqref{ch1:Gr2}. Now, let us define $T_0:H^{-1/2}(\Gz)^3\rightarrow\widetilde{H}^{1/2}(\Gz)^3$ by $T_0(\tbold_{\g})=[\ubold_0]$, where $\ubold_0\in H^1(B_R\setminus\overline{\Gz})^3$ satisfies
\begin{equation}
A_0(\ubold_0,\vbold) = \int_{\Gz} \tbold_{\g}(\ybold) \cdot[\overline{\vbold}](\ybold)\dsy,\qquad\forall\vbold\in H^1(B_R\setminus\overline{\Gz})^3.
\end{equation}
It is worth mentioning that $\ubold_0$ is governed by
\begin{equation}
\begin{aligned}
&\Delta^*\ubold_0(\bxi) - \omega^2 \exs \ubold_0(\bxi)~=~\bold{0}, & \quad \bxi \in B_R\backslash \{\overline{\G \cup \G_1}\},\\
&\dnu \ubold_0(\bxi)~=~\nxs- \tbold_{\g}(\bxi), & \bxi \in \Gz.
\end{aligned}
\end{equation} 
Since $A_0$ is coercive and self-adjoint, $T_0$ is well-defined, bounded and self-adjoint. Moreover, by invoking the trace theorem, one obtains
\begin{equation}
\begin{aligned}
|\nxs\nxs|\tbold_{\g}|\nxs\nxs|^2_{\widetilde{H}^{-1/2}(\Gz)^3}\!=|\nxs\nxs|\dnu \ubold_0|\nxs\nxs|^2_{\widetilde{H}^{-1/2}(\Gz)^3} & \leq C\left[ |\nxs\nxs|\Delta^{\! *}\ubold_0|\nxs\nxs|^2_{L^2(B_R)^3} + |\nxs\nxs|\CC\nxs:\!\nabla\ubold_0|\nxs\nxs|^2_{L^2(B_R)^3} \right]\\*[1mm]
& \leq  \widetilde{C}\left[ |\nxs\nxs|\ubold_0|\nxs\nxs|^2_{L^2(B_R)^3} + |\nxs\nxs|\nabla\ubold_0|\nxs\nxs|^2_{L^2(B_R)^3} \right] \\*[1mm]
&\leq \widetilde{C}_1 |A_0(\ubold_0,\ubold_0)| = \widetilde{C}_1 \bigg\rvert \int_{\Gz}\tbold_{\g}\!\cdot\overline{T_0 (\tbold_{\g})}\dsy\bigg\rvert.
\end{aligned}
\end{equation}
\vspace{-2mm}
Thus, $T_0$ is coercive, i.e.
\begin{equation}\label{CoT}
||\tbold_{\g}||_{\widetilde{H}^{-1/2}(\Gz)^3} \leq \widetilde{C}_1 |\langle T_0 (\tbold_{\g}),\tbold_{\g}\rangle_\Gz |.
\end{equation}
On the other hand, note that if $T_c:=T-T_0$, then by definition $T_c(\tbold_{\g}) = [\ubold_c]$ where $\ubold_c=\ubold-\ubold_0$ satisfies the variational form
\begin{eqnarray}\label{prob-uc}
A_0(\ubold_c,\vbold)=-B(\ubold,\vbold),  \quad \forall\vbold\in H^1(B_R\setminus\overline{\Gz})^3.
\end{eqnarray}
Since $A_0$ is coercive and $B$ is compact, then the mapping $\ubold\mapsto\ubold_c$ is compact. Also, from the trace theorem and well-posedness of (\ref{pr1}) governing $\ubold$, the mapping $\tbold_{\g}\mapsto [\ubold_c]$, i.e. the operator $T_c$, is also compact. The compactness of $\Re(T_c)$ is then an immediate consequence, and therefore the proof is complete.
\end{proof}
\begin{lemma}\label{lemmaK}
Assuming $\KK\in L^\infty(\Gz)^3$ is such that $\Im \langle \KK\boldsymbol{\eta}, \boldsymbol{\eta} \rangle_{\Gz} \leq 0$, $\forall \boldsymbol{\eta} \in \widetilde{H}^{1/2}(\Gz)^3$. Then, the operator $\Im(T)=(T-T^*)/{2i}$ is positive definite, i.e.,
\begin{equation}\label{ImT}
\langle \Im(T)(\tbold_{\g}),\tbold_{\g}\rangle_\Gz > 0, \qquad\forall  \tbold_{\g}\in H^{-1/2}(\Gz)^3.
\end{equation}
\end{lemma}
\begin{proof}
One may write the weak formulation for the scattered signature of $\Gz$ as 
\begin{eqnarray}
A(\tilde{\ubold}_{\text{sc}},\tilde{\ubold}_{\text{sc}}) &=& \int_{\Gz}\tbold_{\g}\cdot \overline{[\us]}\dsy, \qquad \tbold_{\g}\in H^{-1/2}(\Gz)^3,
\end{eqnarray}
where $\tilde{\ubold}_{\text{sc}}=\us\nxs-\nxs\ubold_{b,\g}^{sc}$, with reference to \eqref{ch1:Gr2} and \eqref{bgsc}, and $A(\cdot,\cdot)$ is defined by (\ref{def-A}). Therefore, for a given $\tbold_{\g}\in H^{-1/2}(\Gz)^3$
\begin{equation}
\begin{aligned}
\frac{1}{2i}\big[\langle T(\tbold_{\g}) &,\tbold_{\g}\rangle_\Gz - \langle\tbold_{\g},T(\tbold_{\g})\rangle_\Gz\big]~=~\frac{1}{2i}\left(\int_{\Gz}\overline{\tbold_{\g}}\cdot [\us]\dsy - \int_{\Gz}\tbold_{\g}\cdot \overline{[\us]}\dsy\right)\nonumber\\*[1.5mm]
& =\,-\Im(A(\tilde{\ubold}_{\text{sc}},\tilde{\ubold}_{\text{sc}}))~=\,-\Im\langle \KK[\us],{[\us]} \rangle_{\Gz} \nonumber+ \Im\langle \Top(\tilde{\ubold}_{\text{sc}}), \tilde{\ubold}_{\text{sc}} \rangle_{\Gz} \,>\, 0,
\end{aligned}
\end{equation}
where the last inequality invokes Lemma \ref{lemmaDtN}.
\end{proof}
\noindent For a given trial fracture $L\subset\G$ endowed with an opening displacement profile $\boldsymbol{\eta}\in\widetilde{H}^{1/2}(L)^3$, let us define the signature far-field $\boldsymbol{\phi}_L^\infty\in L^2(\Sp^2)^3$ by
\begin{equation}\label{testf}
\boldsymbol{\phi}_L^\infty(\hat\bxi) \,= \int_L \dnu\Gr_b^\infty(\hat\bxi,\ybold)\boldsymbol{\eta}(\ybold)\!\dsy \,= \int_L \dnu\bold{W}_b(\ybold,-\hat\bxi)\boldsymbol{\eta}(\ybold)\!\dsy, \quad \forall \hat\bxi \in \Sp^2.
\end{equation}
\begin{lemma}\label{ch2lemmaG}
The operator $\Gop:={\mathcal S}_b\Hop^*T$ is such that, $\boldsymbol{\phi}_L^\infty\in Range(\Gop)$ for all $\boldsymbol{\eta}\in \widetilde{H}^{1/2}(L)^3$ not vanishing identically on any subset of $L$ with positive Lebesgue measure, if and only if $L\subset\Gz$.
\end{lemma}
\begin{proof}
By definition, $\Gop: H^{-1/2}(\Gz)^3\rightarrow L^2(\Sp^2)^3$ is a map such that $\Gop\tbold_{\g}=\us^\infty-\ubold_{b,\g}^\infty$, where $\us^\infty$ (resp.~$\ubold_{b,\g}^\infty$) is the far-field pattern affiliated with $\us$ in \eqref{ch1:Gr2} (resp.~$\ubold_{b,\g}^{sc}$ in \eqref{bgsc}). Assuming $L\subset\Gz$, the support of any admissible fracture opening displacement $\boldsymbol{\eta}\in\widetilde{H}^{1/2}(L)^3$ defined over $L$ may be extended to $\Gz$ via zero-padding. Thus-obtained displacement profile is denoted by $\widetilde{\boldsymbol{\eta}} \in \widetilde{H}^{1/2}(\Gz)^3$. In this setting, the far-field pattern $\boldsymbol{\phi}_L^\infty \in L^2(\Sp^2)^3$ is associated to the scattered wavefield $\bw \in H^1_{loc}(\R^3\!\setminus\!\overline{L})^3$
\begin{equation}\label{bwxi}
\bw(\bxi) \,= \int_L \dnu\Gr_b(\bxi,\ybold)\boldsymbol{\eta}(\ybold)\!\dsy \,= \int_{\Gz} \dnu\Grb(\bxi,\ybold)\widetilde{\boldsymbol{\eta}}(\ybold)\dsy, \quad \bxi \in \R^3\!\setminus\!\overline{L},
\end{equation}
which satisfies:
\vspace{0 mm}
\begin{equation} \label{ch4:Gr2}
\begin{aligned}
&\Delta^{\! *} \exs \bw(\bxi) \,+\, \rho(\bxi) \exs \omega^2 \bw(\bxi) \,=\,   \bold{0}, \quad & \bxi \in \R^3\backslash \{\overline{\G \cup \G_1}\}, \\*[0.5 mm]
&\dnu \bw^- (\bxi) \,+\, \dnu \bw^+(\bxi) \,=\, \bold{0}, \,\,  \, & \bxi \in \G,  \\*[0.5 mm]
&[\bw](\bxi) \,=\, [\bw^+ \nxs-\exs \bw^-](\bxi) \,=\,  \widetilde{\boldsymbol{\eta}}(\bxi), \,& \bxi \in \G_0,  \\*[0.5 mm]
&\bw^- (\bxi) \,=\, \bw^+(\bxi),& \bxi \in \G\backslash\overline{\G_0},   \\*[0.5 mm]
&\dnu \bw^+(\bxi) \,+\, \dnu \bw^{\circ}(\bxi) \,=\, \bold{0}, \,\, \bw^+(\bxi) \,=\, \bw^{\circ}(\bxi),  \,& \bxi \in \G_1, 
\end{aligned}
\end{equation}
Complemented by the Kupradze radiation condition at infinity. Thus, if we define $\tbold_{\g}:= \KK\widetilde{\boldsymbol{\eta}}-\dnu \bw$, then $\tbold_{\g}\in H^{-1/2}(\Gz)^3$ and $\Gop(\tbold_{\g})=\boldsymbol{\phi}_L^\infty$.\\
Now, let us assume contrary to the Theorem's statement that there exists $\boldsymbol{\eta}\in \widetilde{H}^{1/2}(L)^3$ such that~(a)~$\boldsymbol{\eta}$ does not vanish in any subset of $L$ of positive Lebesgue measure,~(b)~$\boldsymbol{\phi}_L^\infty$ is within $Range(\Gop)$, and~(c)~$L\not\subset\Gz$. Then, by definition of $\Gop$, $\boldsymbol{\phi}_L^\infty$ is the far-field pattern of $\bw\in H^1_{loc}(\R^3\setminus\overline{\Gz})^3$ in \eqref{bwxi} that satisfies the defective problem \eqref{ch4:Gr2} for some $\tbold_{\g}\in H^{-1/2}(\Gz)^3$.Therefore, one may interpret $\boldsymbol{\phi}_L^\infty$ as the far-field pattern of the two potentials, namely:
\begin{equation}
P_L \boldsymbol{\eta}= \int_{L} \dnu\Grb(\cdot,\ybold)\boldsymbol{\eta}(\ybold)\dsy, \qquad \boldsymbol{\eta}\in\widetilde{H}^{1/2}(L)^3,
\end{equation}
and
\begin{equation}
 \bw = \int_{\Gz} \dnu\Grb(\cdot,\ybold)[ \bw ](\ybold)\dsy, \qquad [ \bw ] \in \widetilde{H}^{1/2}(\Gz)^3.
\end{equation}
By Rellich's lemma and the unique continuation principle, it is known that both potentials are identical in $\R^3\setminus(L\cup\Gz)$. However, by assumption, there exists $\xbold\in L$ and an open neighborhood $V_{\delta}$ of $\xbold$ such that on $V_{\delta}\cap L\subset (L\setminus\overline{\Gz})$ where the density $\boldsymbol{\eta}$ does not vanish. Hence, the potential $P_L$ has a discontinuity on $\xbold$ along the normal direction to $L$, whereas $\bw$ is continuous at the same point, and this is a contradiction.
\end{proof}
\noindent The following corollary follows immediately from Lemma \ref{ch2lemmaG}, and the properties of scattering operator ${\mathcal S}_b$ defined by \eqref{S_b} (see Remark~\ref{SSI}).
\begin{corollary}\label{Coll1}
The operator $\Hop^*$ is such that, ${\mathcal S}_b^*\boldsymbol{\phi}_L^\infty\in Range(\Hop^*)$ for all $\boldsymbol{\eta}\in \widetilde{H}^{1/2}(L)^3$ not vanishing identically on any subset of $L$ of positive Lebesgue measure, if and only if $L\subset\Gz$.
\end{corollary}
\noindent We are now in position to establish the main Theorem of FM, given by Theorem~\ref{theo-FM}, catering for a non-iterative elastodynamic reconstruction of interfacial fractures in heterogeneous composites. 
\begin{theorem}\label{theo-FM}
Under Assumption~\ref{ass1}, the following holds for operator $\widetilde{F}_D:={\mathcal S}_b^*\Fop_D$:\vspace{1 mm}
\begin{itemize}
\item [(1)]  The operator $(\widetilde{F}_D^\#)^{1/2}\! : L^2(\Sp^2)^3 \rightarrow L^2(\Sp^2)^3$ such that $\widetilde{F}_D^\# := |\Re(\widetilde{F}_D)| + \Im(\widetilde{F}_D)$ is positive, and its range coincides with that of  $\Hop^* : \widetilde{H}^{1/2}(\Gz)^3\rightarrow L^2(\Sp^2)^3$.\vspace{1 mm}
\item [(2)] ${\mathcal S}_b^*\boldsymbol{\phi}_L^\infty\in Range((\widetilde{F}_D^\#)^{1/2})$ for all $\boldsymbol{\eta}\in \widetilde{H}^{1/2}(L)^3$ such that $\boldsymbol{\eta}$ does not vanish identically in any subset of $L$ of positive Lebesgue measure, if and only if $L\subset\Gz$.
\end{itemize}
\end{theorem}
\begin{proof}
This theorem is a direct consequence of the abstract Theorem 2.15 in~\cite{kirsch}, and its recent generalization in~\cite[Theorem~3.2]{crack1}. For clarity, we synthesize the results using the present notation. With reference to \eqref{embed}, one may note that $\widetilde{H}^{1/2}(\Gz)^3 \subset H^1_{loc}(\R^3\backslash \Gz)^3 \subset {H}^{-1/2}(\Gz)^3$ is a Gelfand triple involving Hilbert (thus reflexive) spaces with dense embeddings. Moreover, $\widetilde{F}_D :L^2(\Sp^2)^3\rightarrow L^2(\Sp^2)^3$, $\Hop^*:\widetilde{H}^{1/2}(\Gz)^3\rightarrow L^2(\Sp^2)^3$, and $T :{H}^{-1/2}(\Gz)^3 \rightarrow \widetilde{H}^{1/2}(\Gz)^3$ are bounded linear operators according to Lemmas~\ref{lemmaH} and~\ref{lemmaT} such that 
$$\widetilde{F}_D~=~\Hop^*T\Hop.$$
Additionally, the integral operator $\Hop^*:\widetilde{H}^{1/2}(\Gz)^3\rightarrow L^2(\Sp^2)^3$ in \eqref{eqH*bis}:~(a)~is compact owing to its continuous kernel, and~(b)~has a dense range in view of the injectivity of its adjoint operator $\Hop: L^2(\Sp^2)^3 \rightarrow {H}^{-1/2}(\Gz)^3$, see Assumption~\ref{ass1}. In Lemmas~\ref{lemmaT} and~\ref{lemmaK}, it is shown that:~(a)~the operator $T :{H}^{-1/2}(\Gz)^3 \rightarrow \widetilde{H}^{1/2}(\Gz)^3$ defined over the fracture interface $\Gz$ can be decomposed into a compact part $T_c$ and a self-adjoint part $T_0$,~(b)~$T_0$ is coercive according to~\eqref{CoT}, and~(c)~the imaginary part $\Im{T}$ is positive on its domain as described in~\eqref{ImT}. Based on the above arguments, all the conditions for~\cite[Theorem~3.2]{crack1} is established which reads Part~(1)~of Theorem~\ref{theo-FM}. The last part of the Theorem's statement is a direct consequence of Part~(1)~and Corollary~\ref{Coll1}.  
\end{proof}
From Picard's criterion, Theorem 2.7 in \cite{CC06}, the following result is an immediate consequence. 
\begin{corollary}\label{coroll-FM}
Let $\{\mu_\ell,\boldsymbol{\psi}_\ell\}_{\ell=1}^\infty$ be the eigensystem  of $\widetilde{F}_D^\#$,
 then:  $L\subset\Gz$ if and only if 
 \begin{eqnarray}
 \sum_{\ell=1}^\infty \frac{|( \widetilde{\boldsymbol{\phi}_L^\infty} ,\boldsymbol{\psi}_\ell)_{L^2(\Sp^2)^3}|^2}{|\mu_\ell|} &<&\infty,
 \end{eqnarray}
where $\widetilde{\boldsymbol{\phi}_L^\infty}:={\mathcal S}_b^*\boldsymbol{\phi}_L^\infty$  and the density $\boldsymbol{\eta}\in \widetilde{H}^{1/2}(L)^3$ in the definition \eqref{testf} of $\boldsymbol{\phi}_L^\infty$  is such that $\boldsymbol{\eta}$ does not vanish identically in any subset of $L$ of positive Lebesgue measure. 
\end{corollary}
\noindent {\it FM criteria for imaging interfacial damage in layered composites.}~On the basis of Theorem~\ref{theo-FM} and Corollary~\ref{coroll-FM}, a fast yet robust FM-based criterion for the elastic-wave reconstruction of (heterogeneous) interfacial anomalies (such as $\Gz$) can be designed as 
\begin{equation}\label{IFM}
I^{\mathcal{F}}(L) := \frac{1}{|\nxs\nxs|\g^L|\nxs\nxs|_{L^2(\Sp^2)^3}},
\end{equation}
where $\g^L$ is the solution to the far-field equation
\begin{eqnarray}\label{farfeq}
(\widetilde{F}_D^\#)^{1/2}\g^L = \widetilde{\boldsymbol{\phi}_L^\infty}, \qquad \widetilde{\boldsymbol{\phi}_L^\infty}:={\mathcal S}_b^*\boldsymbol{\phi}_L^\infty,
\end{eqnarray}
for all possible open trial fractures $L$ in the sampling region. One may invoke the well-known Tikhonov regularization (see~\cite{fatemeh} for further details) to construct an approximate solution for \eqref{farfeq}, notwithstanding of whether or not $\widetilde{\boldsymbol{\phi}_L^\infty}$ is within the range of $(\widetilde{F}_D^\#)^{1/2}$. Note that the imaging indicator $I^{\mathcal{F}}$ reaches its highest values whenever $L$ approaches a (stationary or advancing) fracture $\Gz$ in the sampling grid. It must be mentioned that in most engineered composites such as metamaterials and reinforced concrete sections used in critical infrastructure, the loci of bi-material interfaces are by-design known. In such situations, the brute-force sampling over a 3D space can be significantly optimized by reducing the search space to a well-defined 2D grid covering the material interfaces corresponding to the original design (of a composite structure) i.e.~in this framework, $L \subset \G$. It should be mentioned that $(\widetilde{F}_D^\#)^{1/2}$ in~\eqref{farfeq} is directly constructed from differential experimental data $\Fop_D = \Fop - \Fop_b$ defined in Section~\ref{sec33}. More importantly, $(\widetilde{F}_D^\#)^{1/2}$ is computed only once and independent of the location and geometry of the trial fracture $L$, so that in computing $I^{\mathcal{F}}$ (over the entire sampling grid), the right-hand-side of \eqref{farfeq} is the only quantity to be re-evaluated at every sampling point according to Section~\ref{sec4}. This remarkably expedites the data inversion process. 
\begin{remark}
The proposed FM indicator, constructed based on differential measurements, provide {\emph{high-resolution images}} of (stationary or evolutionary) interfacial fractures irrespective of the illumination frequency. The cost functionals associated to such indicators, e.g.~obtained via Tikhonov regularization, are by-design convex ~\cite{fatemeh,pour2017}, so that their minimizer can be obtained {\emph{non-iteratively}} which is a major stride toward real-time imaging.
\end{remark}

\begin{remark}
It is worth noting that the FM characterization of $\Gz$ from far-field data (via the range of $(\widetilde{F}_D^\#)^{1/2}$) is rooted in deep geometrical considerations, so that the fracture indicator functionals~(\ref{IFM}) may exhibit only a minor dependence on the heterogeneous nature of the elastic contact at the interface of a hidden fracture -- given by the distribution of~$\KK$ on~$\Gz$. This behavior can be traced back to Lemma~\ref{ch2lemmaG} and Corollary~\ref{Coll1}, where the opening displacement profile $\boldsymbol{\eta} \in \tilde{H}^{1/2}(L)$ -- intimately related to the interface law -- is deemed arbitrary (within the constraints of admissibility). This quality makes the FM imaging paradigm particularly attractive in situations where the fracture's contact law is unknown beforehand, which opens up possibilities for the sequential geometrical reconstruction and interfacial characterization of partially-closed fractures e.g.~\cite{pour2017}.   
\end{remark}

%\vspace{-1 cm}
\begin{remark}\label{remk-generalization}
FM fracture indicator~\eqref{IFM}~naturally lends itself to imaging advanced damage states shown in Fig.~\ref{figrem}, where delamination cracks have branched into the material layers e.g.~$\Omega_-\!$ or $\Omega_+$. In such cases, $\Omega$ may be recast as $\overline{\cup_{\ell=1}^N\Omega_\ell}$ where $\{\Omega_\ell\}_{\ell=1}^N$ is a set of connected domains, with Lipschitz continuous boundaries, whose material properties $\mu_\ell:=\mu\big\rvert_{\Omega_\ell}$, $\lambda_\ell:=\lambda\big\rvert_{\Omega_\ell}$, and $\rho_\ell:=\rho\big\rvert_{\Omega_\ell}$, are continuous. Moreover, adjacent domains $\Omega_\ell$ and $\Omega_j$, such that $\G_{\ell,j}:=\partial \Omega_\ell\cap\partial \Omega_j\neq \emptyset$, satisfy the monotonicity condition \eqref{mon-cond}, which in this context reads:
\begin{equation}
(\lambda_\ell\big\rvert_{\G_{\ell,j}}-\lambda_j\big\rvert_{\G_{\ell,j}})
(\mu_\ell\big\rvert_{\G_{\ell,j}}-\mu_j\big\rvert_{\G_{\ell,j}})\ge 0, \qquad \ell \neq j \in 1,2,\cdot\!\cdot\!\cdot, N.
\end{equation} 
\end{remark}
\noindent In this setting, it is straightforward to show that Theorem~\ref{theo-FM} and Corollary~\ref{coroll-FM} hold using similar arguments. Thus, the FM criteria~\eqref{IFM} remains valid for imaging interfacial damage grown to advanced stages.

 \begin{figure}[tp!]
%\vspace*{4cm}
\begin{center}
\resizebox{0.7\textwidth}{!}{\includegraphics{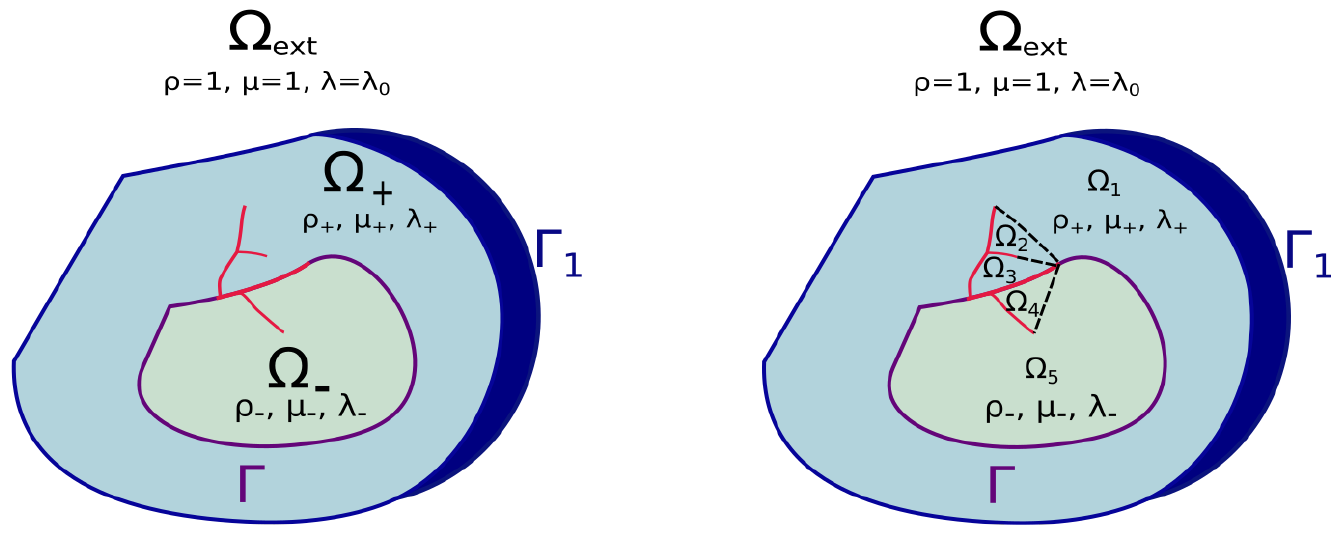}}
\caption{\small{Generalized framework for FM characterization of advanced interfacial damage:~(right)~a generic damage configuration where the interfacial fracture $\Gz$ has propagated and branched into the composite layers $\Omega_-\!$ and $\Omega_+$, and~(left)~mathematical descretization of the composite domain $\{\Omega_\ell\}_{\ell=1}^5$ in order to accommodate different branches of $\Gz$ in FM analysis such that all sub-domains $\{\Omega_\ell\}_{\ell=1}^5$ have Lipschitz continuous boundaries (a geometrical conditions required by FM indicator)}.\label{figrem}
}\vspace*{-5mm}
\end{center}
\end{figure}

%-------------------------------------------------------------------------------------------------------------------------------------------------------------------
\section{\textcolor{black}{Computational validation}}\label{sec4}
%-------------------------------------------------------------------------------------------------------------------------------------------------------------------
This section illustrates the performance of FM indicator functional~\eqref{IFM} for imaging delamination cracks in layered composites through a set of numerical experiments. In what follows the differential sensory data, namely the far-field patterns~\eqref{Fop} over the unit sphere of observation angles $\Sp^2$, are synthetically generated by way of an elastodynamic boundary integral method introduced in~\cite{pour2015,fatemeh-thesis}. 

\noindent {\it Testing configuration.} With reference to Fig~\ref{ch2example}, two distinct composite materials are considered:~\underline{Composite I} (left panel), designed according to the theoretical framework of Fig.~\ref{fig_elastic1}, consists of three homogeneous, isotropic, elastic layers with interfacial boundaries $\G$ and $\G_1$. The exterior layer $\Omega_{ext}$, bounded by an ellipsoidal surface $\G_1$ centered at the origin with semi-axes $(4.5, 4, 6)$, is endowed with non-dimensional shear modulus $\mu_0=1$, Lame constant $\lambda_0=1.5$, and mass density $\rho_0=1$. 
 \begin{figure}[tp!]
\begin{center}
\vspace*{-0.3cm}
\resizebox{0.85\textwidth}{!}{\includegraphics{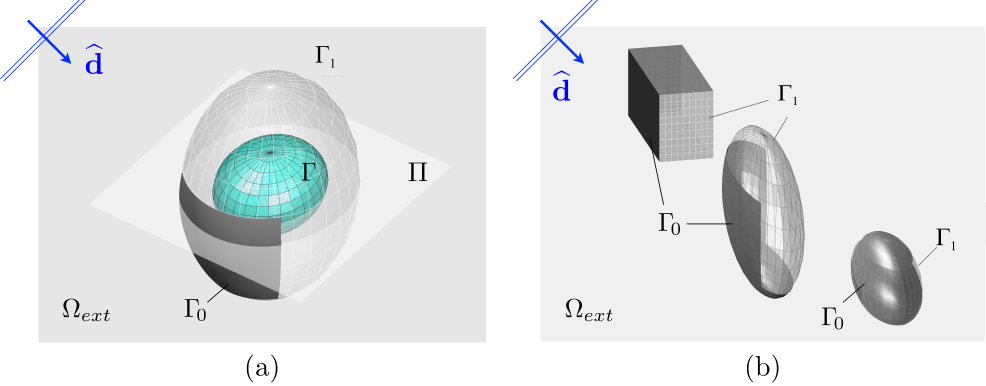}}\vspace*{-0.3cm}
\caption{\small Two elastic-wave sensing setups:~(a)~three-layer composite -- consistent with the theoretical framework of Fig.~\ref{fig_elastic1}, where the bi-material interfaces (denoted by $\G$ and $\G_1$) are ellipsoidal, and the damage zone $\Gz\subset\G_1$ is indicated by dark shade,~and~(b)~two-layer composite with designed material interfaces $\G_1$ and damaged areas $\Gz\subset\G_1$ (in dark shade). Both domains are illuminated by a set of plane P- and S-waves and thus-induced scattered waves are measured in the far field. The sampling grid for FM imaging consists of bi-material boundaries i.e.~$\G \cup \G_1$ in~(a) and $\G_1$ in~(b).}\label{ch2example}\vspace*{-0.5cm}
\end{center}
\end{figure}
The core layer (corresponding to $\Omega_-$ in Fig.~\ref{fig_elastic1}) is an ellipsoid centered at the origin with semi-axes $(3, 2.5, 2)$ and scaled material properties $\mu_-=0.4$, $\lambda_-=0.6$, and $\rho_-=1.5$. The mid-layer (affiliated with $\Omega_{+}$ in Fig.~\ref{fig_elastic1}) is characterized by the relative properties $\mu_+=0.2$, $\lambda_+=0.4$, and $\rho_+=0.75$. In this setting, the shear and compressional wave speeds reads:~(a)~$c_s^0 = 1$ and $c_p^0 = 1.87$ in $\Omega_{ext}$,~(b)~$c_s^+ = 0.52$ and $c_p^+ = 1$ in $\Omega_{+}$, and~(c)~$c_s^- = 0.52$ and $c_p^- = 0.97$ in $\Omega_{-}$. The interfacial damage $\Gz \subset \G_1$ (shaded surface) carries a uniform elastic stiffness
$$\KK(\boldsymbol{x}) \::=\:  (\boldsymbol{\nu} \otimes \bnu)(\boldsymbol{x}) \,+\, (\boldsymbol{\tau}_1 \nxs \otimes \boldsymbol{\tau}_1)(\boldsymbol{x}) \,+\, (\boldsymbol{\tau}_2 \nxs \otimes \boldsymbol{\tau}_2)(\boldsymbol{x}), \qquad \boldsymbol{x} \in \Gz,$$
where $\{\boldsymbol{\tau}_1,\boldsymbol{\tau}_2\}$ indicates an orthonormal set of vectors that span the tangent plane at a point $\boldsymbol{x}$ on $\Gz$, then the vectors $\{\boldsymbol{\tau}_1,\boldsymbol{\tau}_2,\bnu\}$ are a local basis that is well defined everywhere on $\Gz$. On the other hand,~\underline{Composite II} (right panel in Fig.~\ref{ch2example}) is designed to demonstrate the capability of FM indicator in simultaneous reconstruction of multiple interfacial fractures. The background domain consists of two layers:~(i)~exterior domain $\Omega_{ext}$, and~(ii)~three geometrically distinct inclusions with identical material properties -- namely, a cube of side $1.8$ centered at $(0,3,3)$; a sphere of radius $2$ centered at $(0,-4,-2)$; and an ellipsoid with semi-axes $(3, 2, 4)$ centered at the origin. \emph{Layer~(i)}~is a homogeneous, isotropic elastic solid with nondimensionalized shear modulus $\mu_0=1$, Lame constant $\lambda_0=1.5$, and mass density $\rho_0=1$. \emph{Layer~(ii)}~is endowed with uniform (relative) properties $\mu_+=0.2$, $\lambda_+=0.4$, and $\rho_+=0.5$. In this case, the shear and compressional wave speeds are:~(a)~$c_s^0 = 1$ and $c_p^0 = 1.87$ in~\emph{Layer~(i)}, and~(b)~$c_s^+ = 0.63$ and $c_p^+ = 1.26$ in~\emph{Layer~(ii)}. The  bi-material interfaces in the background domain is denoted by $\G_1$, and the support $\Gz \subset \G_1$ of interfacial damage is illustrated by a dark shade. The contact stiffness $\KK$ affiliated with $\Gz$ is given in the fractures local coordinates $\{\boldsymbol{\tau}_1,\boldsymbol{\tau}_2,\bnu\}$ by
\begin{equation}\nonumber
\KK(\boldsymbol{x}) ~:=~ \left\{\begin{array}{cc} \hspace{-7.7 cm}  \boldsymbol{0} & \hspace{-0.7 cm} \boldsymbol{x} \,\,\text{on ellipsoid}  \\ 2(\boldsymbol{\nu} \otimes \bnu)(\boldsymbol{x}) \,+\, 2(\boldsymbol{\tau}_1 \nxs \otimes \boldsymbol{\tau}_1)(\boldsymbol{x}) \,+\, 2(\boldsymbol{\tau}_2 \nxs \otimes \boldsymbol{\tau}_2)(\boldsymbol{x}) & \quad \boldsymbol{x} \,\,\text{on cube or sphere}   \end{array}\right.
\end{equation}

\noindent The designated domains in Fig.~\ref{ch2example} are interrogated by elastic waves in two steps, where:~(1)~the background domain is illuminated by (P- and S-) waves propagating in direction $\widehat\bd$, and~(2)~the probing campaign is repeated after the formation of interfacial damage $\Gz$. In \emph{Step 1}, the interaction of $\G$ and $\G_1$ with incident waves gives rise to the total wavefield $\ubold_b$ solving~\eqref{pr1bb} -- whose associated (scattered) far-field pattern $\ubold^\infty$ is computed on the basis of decomposition~(\ref{ffp}). In \emph{Step 2}, the interaction of $\G$, $\G_1$, and $\Gz$ with the same illuminating waveforms gives birth to the total wavefied \eqref{pr1}, possessing a far-field expansion of the form~(\ref{ffp}). 

\vspace{-1mm}
\begin{remark} 
In absence of direct measurements in \emph{Step 1}, one may synthetically solve~\eqref{pr1bb} to retrieve the background scattered field $\ubold^{sc} = \ubold_b - \ubold^i$, whose far-field pattern is used to construct $\Fop_D$ required for computing the FM imaging indicator~\eqref{IFM}.  
\end{remark}
\vspace{-1mm}

\noindent \emph{The discretized far-field operator.}~For both illumination and sensing purposes, the unit sphere $\Sp^2$ is sampled by a uniform grid of $N_\theta \!\times\! N_\phi$ observation directions, specified by the polar ($\theta_j,\, j\!=\!1,\ldots N_\theta$) and azimuthal~($\phi_k, \,k\!=\!1,\ldots N_\phi$) angle values. In this setting, the discretized far-field operator, associated to~\eqref{farfeq}, is represented as a $3N\!\times 3N$ matrix ($N\!=\!N_\theta N_\phi$) constructed as the following   
\begin{equation}\label{mat0}
\widetilde{\text{\bf F}}_{\!D}^\#= \big|\Re(\widetilde{\text{\bf F}}_{\!D})\nxs\big| + \Im(\widetilde{\text{\bf F}}_{\!D}), \quad  (\widetilde{\text{F}}_{\!D})_{ k \ell} =  (\overline{\text{S}}_b)_{\ell k}( \textrm{{F}} - \textrm{{F}}_b)_{ k \ell}, \qquad \,\, k,\ell = 0,\ldots 3N-1,
\end{equation}
where for $i,j = 0,\ldots, N-1$ and $\text{i} = \sqrt{-1}$, 
\begin{equation}\label{DF}
\begin{aligned}
&(\textrm{\bf{F}}-\textrm{\bf{F}}_b)(3i\nxs+\nxs1\!:\!3i\nxs+\nxs3, \,3j\nxs+\nxs1\!:\!3j\nxs+\nxs3) \,=\,(\textrm{\bf{W}}^\infty\! -\textrm{\bf{W}}_b^\infty) (\hat\bxi_i,\widehat\bd_j), \\*[0.5 mm]
&\textrm{\bf{S}}_b(3i\nxs+\nxs1\!:\!3i\nxs+\nxs3, \,3j\nxs+\nxs1\!:\!3j\nxs+\nxs3) \,=\, \textrm{\bf I}_{\exs 3\times3} - 2 \textrm{i} (k_p \alpha_p \oplus k_s \alpha_s)\textrm{\bf{W}}_b^\infty (\hat\bxi_i,\widehat\bd_j),
\end{aligned}
\end{equation}
wherein $\textrm{\bf I}$ is the $3 \times 3$ identity matrix;~$(W_b^{\infty})_{rs}$, $\lbrace r,s\!=\!1,2,3\rbrace$ stands for the leading-order far-field expansion of the background response tensor, according to~\eqref{ffp} i.e.,
$$ \textrm{\bf{W}}_b^\infty = \textrm{\bf{W}}_b^{p,\infty} \oplus  \textrm{\bf{W}}_b^{s,\infty}.$$
Here, the superscript \emph{p} (resp.~\emph{s}) indicates affiliation with P-wave (resp.~S-wave) patterns. One should bear in mind that $\textrm{\bf{W}}_b^\infty(\hat\bxi_i,\widehat\bd_j)$ is a linear map taking an arbitrary polarization vector $\bq \in \Sp^2$ (related to an incident plane wave propagating along $\widehat\bd_j \in \Sp^2$) to the far-field pattern $\ubold_b^\infty(\hat\bxi_i) \in L^2(\Sp^2)^3$ affiliated with $\ubold_b - \ubold^i$ (see~\eqref{plwa} and~\eqref{pr1bb}) in the background domain such that $\ubold_b^\infty=\textrm{\bf{W}}_b^\infty \bq$. In this vein, $\textrm{\bf{W}}^\infty\!$ in~\eqref{DF} is a similar map taking $\bq$ in~\eqref{plwa} to the far-field pattern $\ubold^\infty \in L^2(\Sp^2)^3$ in~\eqref{ffp} associated to $\ubold - \ubold^i$ (see~\eqref{pr1}) in the damaged state so that $\ubold^\infty=\textrm{\bf{W}}^\infty \bq$. In what follows, unless stated otherwise, we assume $N_\theta=20$ and $N_\phi=10$. 

\noindent \emph{Noisy data.}~When evaluating $\widetilde{\text{\bf F}}_{\!D}$ and $\textrm{\bf{S}}_b$ in~\eqref{mat0} and~\eqref{DF}, the presence of noise in measurements is accounted for by considering the perturbed operators    
\begin{equation}\label{DFN}
\textrm{\bf{F}}^\delta \,\, \colon \!\!\!= \, (\boldsymbol{I}_{\exs 3N\times3N} + \boldsymbol{N}_{\!\epsilon} ) \exs \textrm{\bf{F}}, \qquad \textrm{\!\bf{F}}_b^{\delta'} \,\, \colon \!\!\!= \, (\boldsymbol{I}_{\exs 3N\times3N} + \boldsymbol{N}_{\!\epsilon'} ) \exs \textrm{\bf{F}}_b,
\end{equation}
where $\boldsymbol{I}$ is the $3N \times 3N$ identity matrix, and $\boldsymbol{N}_{\!\epsilon}$ (resp.~$\boldsymbol{N}_{\!\epsilon'}$) is the noise matrix of commensurate dimension whose components are uniformly-distributed (complex) random variables in $[-\epsilon, \, \epsilon]^2$ (resp.~$[-\epsilon', \, \epsilon']^2$). In the following simulations, the primary and secondary noise measures $\delta' = |\nxs\nxs\nxs|{\boldsymbol{N}_{\!\epsilon'} \exs \textrm{\bf{F}}}_b|\nxs\nxs\nxs|/|\nxs\nxs\nxs|\textrm{\bf{F}}_b|\nxs\nxs\nxs|$ and $\delta = |\nxs\nxs\nxs|{\boldsymbol{N}_{\!\epsilon} \exs \textrm{\bf{F}}}|\nxs\nxs\nxs|/|\nxs\nxs\nxs|\textrm{\bf{F}}|\nxs\nxs\nxs|$ take the value of~$\delta\% = \delta'\% = 5$. 

\noindent \emph{Trial far-field pattern.}  The FM indicator map~\eqref{IFM} is constructed by solving~\eqref{farfeq} for the synthetic source density $\g^L \in L^2(\Sp^2)^3$ over a 2D grid of trial infinitesimal fractures $L=\bx_{\small \circ}\!+\bR{\sf L}$, where~$\bx_{\small \circ}\nxs$ denotes the sampling point and $\bR$ is a unitary rotation matrix. Here, this is accomplished by taking~{\sf L} to be a vanishing penny-shaped fracture with unit normal~$\bn_{\small \circ}\nxs:=\lbrace 0,0,1\rbrace$, i.e.~by setting the fracture opening displacement profile in~(\ref{testf}) as $\boldsymbol{\eta}(\ybold) = \textcolor{black}{\delta (\ybold-\bx_{\small \circ}\!)|{\sf L}|^{-1}} \bR\bn_{\small \circ}\nxs$. Note that the 2D sampling grid, proposed in this work, canvases the support of bi-material interfaces in the \emph{background domain} which are known a-priori e.g.~in the context of Fig.~\ref{ch2example}~(a), the sampling grid covers $\G \cup \G_1$. As a result, the orientation of $L$ at every sampling point $\bx_{\small \circ}\nxs$ can be considered known and equal to the direction of local normal vector to the relevant bi-material interface i.e.~$\bR\bn_{\small \circ}\nxs = \boldsymbol{\nu}(\bx_{\small \circ}\nxs)$ -- where in Fig.~\ref{ch2example}~(a), $\bx_{\small \circ}\nxs \in \G \cup \G_1$. In this setting, one  finds that    
\begin{equation}\label{Nrhs}
\phi^\infty_{\bx_{\small \circ}\nxs}(\hat\bxi) := (\bnu \otimes \bnu)(\bx_{\small \circ}\nxs) : \CC: \nabla \bold{W}_b(\bx_{\small \circ}\nxs,-\hat\bxi),
\end{equation}
where the normal vector $\bnu(\bx_{\small \circ}\nxs)$ and elasticity tensor $\CC$ are known from the background configuration, and $\bold{W}_b(\bx_{\small \circ}\nxs,-\hat\bxi)$ is computed by solving~\eqref{pr1bb} for $\ubold_b$ when the incident wavefield $\ubold^i$ in~\eqref{plwa} is propagating in direction $\bd = - \hat\bxi$ and polarized along the reference $-(\hat\bxi,\hat{\boldsymbol{\theta}},\hat{\boldsymbol{\phi}})$ orthonormal basis i.e.~$\bq_p = -\hat\bxi$ and $\bq_s = -\hat{\boldsymbol{\theta}},-\hat{\boldsymbol{\phi}}$. Discretization of~\eqref{Nrhs} in accordance with~\eqref{DFN} leads to a $3N\times1$ vector 
 \begin{equation}\label{DNrhs}
\boldsymbol{\Phi}^\infty_{\bx_{\small \circ}\nxs}(3i+1:3i+3) = (\bnu \otimes \bnu)(\bx_{\small \circ}\nxs) : \CC: \nabla \bold{W}_b(\bx_{\small \circ}\nxs,-\hat\bxi_i), \qquad i = 0,\ldots, N-1.
\end{equation} 
As a result, the far-field equation~\eqref{farfeq} takes the discretized form 
\begin{equation}\label{Dff}
(\widetilde{\text{\bf F}}_{\!D}^\#)^{\frac{1}{2}} \,\g^\delta_{\bx_{\small \circ}\nxs}\,=\, \textrm{\bf{S}}^*_b \boldsymbol{\Phi}^\infty_{\bx_{\small \circ}\nxs}, 
\end{equation}
where star indicates complex conjugate transpose of $\textrm{\bf{S}}_b$.~\eqref{Dff} provides the basis for computing the FM indicator functionals. 

\noindent \emph{Sampling.} As shown in Fig.~\ref{ch2example}, the search area overlays the intrinsic bi-material interfaces of the background domain where the FM indicator functionals are evaluated. More specifically, in panel~(a)~of the figure, the ellipsoidal search surfaces $\G \cup \G_1$ are collectively probed by $2225$  sampling points $\bx_{\small \circ}\nxs$, while in panel~(b),~the 2D ellipsoidal, spherical and cubical grids are respectively covered by $900$, $200$, and $150$ sampling points. As mentioned earlier, the outward normal vector $\bnu(\bx_{\small \circ}\nxs)$ to each sampling surface indicates the orientation of the trial fracture $L(\bx_{\small \circ}\nxs)$. Accordingly, the FM fracture indicator map is constructed by solving~(\ref{Dff}) for a total of $M = 2225$ (resp.~$M=1250$) trial positions $\bx_{\small \circ}\nxs$ in panel~(a) (resp.~panel~(b)) of Fig.~\ref{ch2example}. The resulting distributions are shown in Fig.~\ref{ex1}~(a) and Fig.~\ref{ex2}~(a). 
  \vspace{-0.5mm}
 \begin{remark}
 It is worth mentioning that the far-field operator $(\widetilde{\text{\bf F}}_{\!D}^\#)^{{1}/{2}}$ -- constructed from measured far-field data, is independent of any particular choice of $L(\bx_{\small \circ}\nxs)$, and thus, remains the same for all $M$ variations of $\boldsymbol{\Phi}^\infty_{\bx_{\small \circ}\nxs}$. Therefore, for computational efficiency, one may recast the right-hand side of (\ref{Dff}) as a $3N \!\times\! M$ matrix $\boldsymbol{\Phi}^\infty$ -- encompassing all choices of ${\bx_{\small \circ}\nxs}$, and solve only one equation to construct the entire 3D indicator map.       
 \end{remark}

\noindent \emph{FM damage indicator.}~At this point, the solution to~(\ref{Dff}) is constructed on the basis of Tikhonov regularization~\cite{Kress1999,fatemeh} i.e.~through the non-iterative minimization
\begin{equation}\label{lssm1}
\g^{\text{\sf \!T}}_{\bx_{\small \circ}\nxs} \,\,\colon \!\!\!= \,\, \text{argmin}_{\g^\delta_{\bx_{\small \circ}\nxs}} \Big\{  |\nxs\nxs\nxs|{(\widetilde{\text{\bf F}}_{\!D}^\#)^{\frac{1}{2}} \,\g^\delta_{\bx_{\small \circ}\nxs}- \textrm{\bf{S}}^*_b \boldsymbol{\Phi}^\infty_{\bx_{\small \circ}\nxs}}|\nxs\nxs\nxs|^2_{L^2(\Sp^2)^3} +  \alpha^{\text{\sf T}}_{\bx_{\small \circ}\nxs} \, |\nxs\nxs\nxs|{\g^\delta_{\bx_{\small \circ}\nxs}}|\nxs\nxs\nxs|^2_{L^2(\Sp^2)^3}\Big\},
\end{equation}
where the regularization parameter $\alpha^{\text{\sf T}}_{\bx_{\small \circ}\nxs}$ is obtained by way of Morozov discrepancy principle~\cite{Kress1999}. 
\vspace{-1mm}
\begin{remark}
In light of Corollary~\ref{coroll-FM}, let $\{\mu_\ell,\boldsymbol{\Psi}_\ell\}_{\ell=1}^{3N}$ be the eigensystem of positive and self-adjoint operator $(\widetilde{\text{\bf F}}_{\!D}^\#)^{\frac{1}{2}}$, then one may also build an approximate solution to~\eqref{Dff} such that  
\begin{equation}\label{ginv}
|\nxs\nxs|\g^{\text{\sf P}}_{\bx_{\small \circ}\nxs}|\nxs\nxs|^2_{L^2(\Sp^2)^3}  := \sum_{\ell=0}^{N_{\text{\sf P}}} \frac{| \textrm{\bf{S}}^*_b \boldsymbol{\Phi}^\infty_{\bx_{\small \circ}\nxs}\!\cdot\! \overline{\boldsymbol{\Psi}_\ell}|^2}{\mu_\ell},
\end{equation}
where $N_{\text{\sf P}}$ is a heuristic truncation level (see Theorem 2.11 in \cite{CC06} for more details).
\end{remark}
\noindent On the basis of~\eqref{lssm1} and~\eqref{ginv}, the FM indicator functional reads
\begin{equation}\label{FiFM}
I^{\mathcal{F}}({\bx_{\small \circ}\nxs}) \,\, := \,\, \frac{1}{|\nxs\nxs|\g^{\text{\sf \exs a}}_{\bx_{\small \circ}}|\nxs\nxs|_{L^2(\Sp^2)^3}}, \qquad \text{\sf \!a} = {\text{\sf T}} \,\text{or} \, {\text{\sf P}}.
\end{equation}
On introducing the map
\[
\mathbbm{1}({\bx_{\small \circ}\nxs}) \,\, := \,\, 
\begin{cases}
1 \quad\quad\text{ if } I^{\mathcal{F}}({\bx_{\small \circ}\nxs}) >\tau_{tol} \times \text{max}(I^{\mathcal{F}}),\\
0 \quad\quad\text{ otherwise,}
\end{cases}
\]
the truncated indicator functional may be expressed as
\begin{equation}\label{TiFM}
I^{\mathcal{F}}_{\tau_{tol}}({\bx_{\small \circ}\nxs}) \,\, := \,\, \mathbbm{1}({\bx_{\small \circ}\nxs}) \, I^{\mathcal{F}}({\bx_{\small \circ}\nxs}). 
\end{equation}
The truncation parameter $\tau_{tol}$ is set to $0.1$ in the sequel. 

\noindent \emph{Results.}~Both composites in Fig.~\ref{ch2example} are illuminated by plane waves with shear wavelength $\lambda_s = 2\pi/4$ that is comparable with the representative length scales of the intrinsic bi-material interfaces $\G \cup \G_1$ and that of $\Gz$. Fig.~\ref{ex1}~(a) illustrates the full-aperture FM reconstruction of $\Gz$ in \emph{Composite I} over the ellipsoidal sampling surface~$\G \cup \G_1$. For clarity, the indicator maps are thresholded by $10\%$, i.e.~only the sampling points whose~$I^{\mathcal{F}}({\bx_{\small \circ}\nxs})$ values are higher than ten percent of the global maximum value -- the support of $I^{\mathcal{F}}_{0.1}$, are shown in Fig.~\ref{ex1}~(b). Also, Fig.~\ref{ex2} shows the corresponding plots for FM reconstruction of $\Gz$ in \emph{Composite II} over the (ellipsoidal, spherical and cubical) sampling grid~$\G_1$.

 \begin{figure}
\begin{center}
\vspace*{-0.2cm}
\resizebox{1\textwidth}{!}{\includegraphics{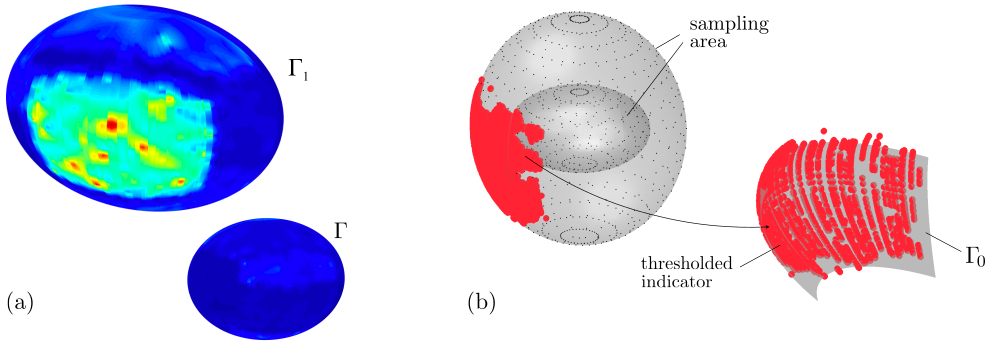}}
\caption{\small Full-aperture FM reconstruction of interfacial fracture $\Gz$ in the three-layer \emph{Composite I}, see Fig.~\ref{ch2example}~(a), over the ellipsoidal sampling surface~$\G \cup \G_1$:~(a)~total FM indicator map $I^{\mathcal{F}}({\bx_{\small \circ}\nxs})$ obtained according to~\eqref{FiFM}, and~(b)~the support of truncated FM indicator $I^{\mathcal{F}}_{0.1}$ computed based on~\eqref{TiFM}.}\label{ex1}
%, with $\tau_{tol}=?$.
\end{center}
\end{figure}
 \begin{figure}
\begin{center}
\vspace*{-0.2cm}
\resizebox{1\textwidth}{!}{\includegraphics{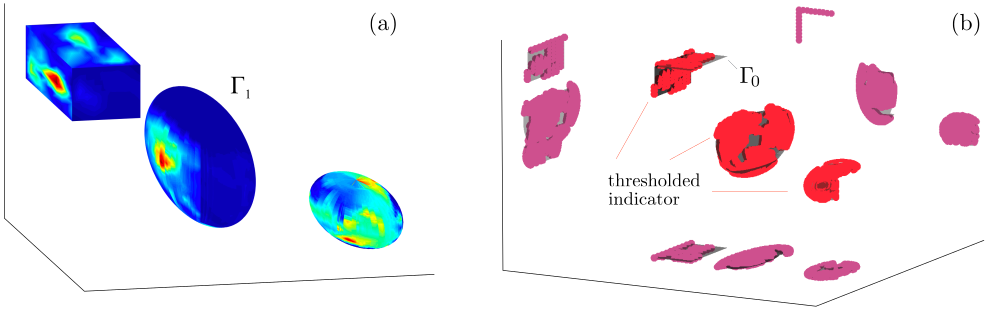}}
\caption{\small Full-aperture FM reconstruction of interfacial damage $\Gz$ in the two-layer \emph{Composite II}, see Fig.~\ref{ch2example}~(b), over a 2D sampling surface~$\G_1$ (comprised of an ellipsoidal, a spherical and a cubical grid):~(a)~total FM indicator map $I^{\mathcal{F}}({\bx_{\small \circ}\nxs})$ according to~\eqref{FiFM}, and~(b)~the support of truncated FM indicator $I^{\mathcal{F}}_{0.1}$ computed based on~\eqref{TiFM}.
}\label{ex2}
\end{center}
\end{figure}

\section*{Conclusion}
The $F_\sharp$-factorization applied to sequential elastodynamic measurements form a fast, yet robust, platform for the geometric reconstruction of damage along bi-material interfaces in layered composites (such as metamaterials and reinforced components used in critical infrastructures). It is shown that the FM-based indicator possesses little sensitivity to reasonable levels of measurement noise, while carrying a top-tier localization property. Such attributes guarantee a high-quality geometric characterization of interfacial fractures notwithstanding the frequency regime of excitation and the unknown (heterogeneous and dissipative) interfacial stiffness $\KK$. The analytical framework presented here accommodates for an \emph{unconnected} fracture support, so that the FM indicator can be utilized for simultaneous reconstruction of multiple fractures in a layered medium, as illustrated in the numerical experiments. The proposed fracture indicator naturally lends itself to imaging complex damage configurations where interfacial cracks in advanced stages have branched into the material layers forming arbitrary-shaped discontinuity networks. In this case, the premises and elements of analysis for applicability of the proposed imaging functional are discussed.

%The method proposed in this work can be naturally extended to a more general class of problems, where not only interfacial cracks, but also transversally intersecting fractures crossing bi-material interfaces can be detected as long as the material properties satisfy the conditions stated in Remark \ref{remk-generalization}.
%}

\section*{Authors' contributions}

I.D.T.~developed Sections 1, 2, and 3.4, and F.P.~obtained the results of Sections 3.2, 3.3 and 4. Both authors equally contributed to the preparation of the manuscript. The data supporting the analysis and conclusions can be accessed by contacting F.P.  

\section*{Acknowledgements}

The corresponding author kindly acknowledges the support provided by the University of Colorado Boulder. This work utilized the RMACC Summit supercomputer, which is supported by the National Science Foundation (awards ACI-1532235 and ACI-1532236), the University of Colorado Boulder, and Colorado State University. The Summit supercomputer is a joint effort of the University of Colorado Boulder and Colorado State University. Special thanks are due to Professor Fioralba Cakoni  for her input during the course of this investigation.

%\bibliographystyle{plain}
%\bibliography{mybib}
\bibliographystyle{siamplain}
\bibliography{mybib}

\end{document}